\documentclass[12pt]{article}
\usepackage{amsthm,amsfonts,amsmath,amssymb}
\usepackage[cp1251]{inputenc}
\usepackage[english, russian]{babel}
\usepackage[final]{graphicx}

\def\R{\Bbb R}

\newtheorem{lemma}{Лемма}
\newtheorem{theorem}{Теорема}
\newtheorem{remark}{Замечание}

\newtheorem{propos}{Предложение}

\begin{document}

\title{\bf Равномерная сходимость \\ на подпространствах в
эргодической теореме фон Неймана с непрерывным временем}

\author{Качуровский А.Г., Подвигин И.В., Тодиков В.Э.}
\date{}
\maketitle

\section{Введение}

{\bf\thesection.1.}
Пусть $\mathcal{H}$ -- гильбертово пространство, $\{U^t\}_{t\in\mathbb{R}}$
--- группа унитарных операторов, действующих в $\mathcal{H}$.
Будем предполагать, что группа~$\{U^t\}_{t\in\mathbb{R}}$ является сильно
непрерывной (см.~\cite[гл.~VIII, \S1]{DSh1}; такими будут, например,
рассматривавшиеся в~\cite{Neum} группы унитарных операторов
~$\{U^t\}_{t\in \R}$ в $L_2(\Omega),$ порождаемые группами
преобразований, сохраняющих меру на вероятностном пространстве
$\Omega$, т.е. потоками).

Для каждого вектора ${f\in\mathcal{H}}$ и всех ${t,s\in\mathbb{R}}$
таких, что ${t>s},$ рассмотрим эргодические средние
$$
P_{t,s}f=\frac{1}{t-s}\int \limits_{s}^{t}U^{\tau}fd\tau.
$$
Классическая эргодическая теорема фон
Неймана~\cite{Neum} утверждает для каждого вектора
${f\in\mathcal{H}}$ существование предела по норме пространства
$\mathcal{H}$
$$
\lim_{t-s\rightarrow\infty}P_{t,s}f=f^*:=Pf,
$$
где $P$ --- ортогональный проектор на подпространство неподвижных
векторов группы~$\{U^t\}_{t\in\mathbb{R}}$.

Для группы~$\{U^t\}_{t\in\mathbb{R}}$ по теореме Стоуна~\cite[гл.~XII,
\S6]{DSh2} существует инфинитезимальный генератор $\mathbf{B}$, т.е.
${U^t=e^{it\mathbf{B}}}$ для всех ${t\in\mathbb{R}.}$ Пусть
${\{E(\lambda)\}_{\lambda\in\mathbb{R}}}$
--- разложение единицы, соответствующее группе
$\{U^t\}_{t\in\mathbb{R}}$ (ее инфинитезимальному генератору). Тогда (см.,
например, \cite{Neum}) для любых векторов ${f,g\in\mathcal{H}}$
справедливо представление
$$
(U^tf,g)_{\mathcal{H}}=\int_{-\infty}^\infty
e^{it\lambda}\,d(E(\lambda)f,g)_{\mathcal{H}}
%\ t\in\mathbb{R},
$$
для всех $t\in\mathbb{R},$ и можно ввести спектральные меры
$\sigma_{f,g}$ --- такие конечные борелевские меры на действительной
прямой, что для всех ее борелевских подмножеств~$C$
$$
\sigma_{f,g}(C)=\int\limits_{C}d(E(\lambda) f,g)_\mathcal{H};
$$
обозначим ${\sigma_{f}=\sigma_{f,f}}.$
Как хорошо известно (и
было доказано и применено еще в \cite{Neum}), при всех $t>s$
справедливо интегральное представление
    $$
    \|P_{t,s}f\|_\mathcal{H}^2
    =  \int    \limits_{-\infty}^{+\infty}
    \left(\frac{\sin\frac{(t-s)x}{2}}{\frac{(t-s)x}{2}}\right)^2d\sigma_{f}(x)
    = \int\limits_{-\infty}^{+\infty} F_{t-s}(x)  \, d\sigma_{f}(x)
    $$
для ядра
$$
F_\tau(x) = \left(\frac{\sin \frac{\tau x}{2}}{\frac{\tau
        x}{2}}\right)^2, \ x\neq 0; \\\ F_\tau(0)=1;
%\ x\in\mathbb{R},
$$
очевидно, всегда $0\leq F_\tau(x) \leq 1$,
и $F_\tau(x) \leq {4 \over (\tau x)^2}$  для всех $x \not = 0$.

Учитывая, что ${P=E(0)},$ при всех $t>s$ получаем интегральное представление
$$
    \|P_{t,s}f-f^*\|_\mathcal{H}^2=\|P_{t,s}(f-f^*)\|_\mathcal{H}^2
    =\int_{\mathbb{R}}
    F_{t-s}(x)  \, d\sigma_{f-f^*}(x)=\int_{\mathbb{R}\setminus\{0\}}
    F_{t-s}(x)  \, d\sigma_{f}(x).\eqno(1)
$$

{\bf\thesection.2.} В недавней работе~\cite{JP} Бен-Арци и Морисс
обнаружили существование (степенной) равномерной сходимости на
некоторых специальных подпространствах в эргодической теореме фон
Неймана с непрерывным временем. А именно, для некоторого банахова
подпространства $\mathcal{X},$ плотно и непрерывно вложенного в
$\mathcal{H},$ была получена оценка
${\|P_{t,-t}-P\|_{\mathcal{X}\to\mathcal{H}}=\mathcal{O}(t^{-l})}$
для некоторого  ${l=l(\mathcal{X})>0}$ при ${t\to\infty}$ (точную
формулировку мы приводим в теореме~\ref{ThMorisse+BenArt}). Поиски
места этого результата в общей теории скоростей сходимости в
эргодических теоремах (см., например, обзоры~\cite{Ka96,KaPo16})
привели к появлению нашей статьи.

В теореме~\ref{Th1} мы даем спектральный критерий равномерной
степенной сходимости с показателем $\alpha\in [0,2)$ на одномерных
подпространствах в $\mathcal{H}$. Оказывается (замечание~\ref{rm2}),
степенная скорость сходимости с показателем $\alpha=2$ является
максимально возможной; теорема~\ref{Th2} дает спектральный же (в
несколько других терминах) критерий наличия такой максимальной
скорости на одномерных подпространствах. Попутно теоремы~\ref{Th1}
и~\ref{Th2} обобщают и уточняют оценки скоростей сходимости в
эргодической теореме фон Неймана для (полу)потоков, полученные ранее
в~\cite{JK}; при этом одна из новых констант этих оценок оказывается
(п.~3.4 раздела 3) точной.

Далее в теоремах~\ref{ThMain} и~\ref{Th2+} мы переносим критерии
теорем~\ref{Th1} и~\ref{Th2} на общий случай многомерных векторных
подпространств в $\mathcal{H}$ со своими нормами. Удалось также
получить (теорема~\ref{ThCharacterizationOfSpaces}) полное описание
всех таких многомерных (нормированных, со своими нормами)
подпространств с равномерной степенной со всеми возможными
показателями ${\alpha\in [0,2]}$ сходимостью.

В разделе 3 мы покажем (замечание~\ref{rm0+}), что рассматриваемая
равномерная сходимость на всем пространстве $\mathcal{H}$ имеет место
лишь в случае спектрального пробела --- что в теореме
фон Неймана для потоков на пространстве Лебега равносильно (замечание~\ref{rm0++})
периодичности п.в. траекторий с ограниченными в совокупности периодами, т.е. бывает
лишь в тривиальных случаях; уже этим объясняется наш интерес к равномерной сходимости
именно на подпространствах в $\mathcal{H}$ со своими нормами.

В разделе 4 рассматриваются также возможные приложения полученных
результатов к исследованию решений некоторых уравнений математической физики,
проведенным в~\cite{JP}.

\section{Одномерные подпространства с равномерной \\ степенной сходимостью}

{\bf\thesection.1.} Полное описание всех одномерных подпространств с равномерной
степенной (с показателями меньше 2) сходимостью фактически было дано в~\cite{KR,JK}, где
%Следующая теорема является аналогом основного результата
%статьи~\cite{JK}, в котором
рассматривались эргодические средние полугрупп изометрических
операторов ~$\{U^t\}_{t\geq 0}$ в $L_2(\Omega),$ порожденных
полугруппами преобразований, сохраняющих меру на вероятностном
пространстве $\Omega$ (т.е. обычная эргодическая теорема фон Неймана
для полупотоков). Поскольку представление~(1) справедливо и для тех
эргодических средних (подробности см., например, в~\cite{JK}), а
константы наших оценок обеих частей теоремы~1 ниже точнее
соответствующих им констант из~\cite{JK}, то доказательство
теоремы~1 немедленно дает и уточнение тех старых оценок.

Основное уточнение получилось за счет другого подхода к
доказательству леммы~1 ниже. В~\cite{JK} для оценки интеграла от
ядра Фейера по спектральной мере использовалось разложение его в ряд
(следуя подходу для случая дискретного времени
из~\cite{Ka96,KS10,KS11}). Мы здесь применяем для оценки этого
интеграла использовавшийся в~\cite{KR} и оказавшийся здесь более
подходящим метод интегрирования по частям, предложенный В.Ф.
Гапошкиным (для случая дискретного времени) в~\cite{Gap}. Отметим
также, что, поскольку представление (1) справедливо и для
(рассматривавшихся в~\cite{Gap}) средних стационарных в широком
смысле процессов с непрерывным временем --- см., например,
\cite[теорема 18.3.1]{IL}, то наша теорема 1 имеет очевидный точный
аналог и для тех стохастических процессов.
%Главные отличия заключаются в том, что мы
%рассматриваем произвольную группу унитарных операторов и, следуя фон
%Нейману, берем двупараметрические усреднения.

Далее нам потребуются следующие две технические леммы, уточняющие
соответствующие им леммы 1 и 2 из~\cite{JK}.

    \begin{lemma}\label{lmUpperEst}
        Для всех $t > s$ справедливо неравенство
        $$
        \|P_{t,s} f- f^{*}\|_\mathcal{H}^2\leq \frac{8}{(t-s)^2}\int \limits_{2\over t-s}^{\infty} x^{-3}\sigma_{f-f^*}(-x,x]dx.
        $$
     \end{lemma}
\begin{proof}[Доказательство леммы~\ref{lmUpperEst}] Без ограничения общности считаем, что  ${f^*=0}.$
Положим ${\tau=t-s}$ и воспользуемся представлением (1):
    $$
    \|P_{t,s}f-f^*\|_\mathcal{H}^2 =  \|P_{t,s}f\|_\mathcal{H}^2 =
    \int\limits_{-\infty}^{+\infty}
    F_{\tau}(x)  \, d(E(x)f,f)_\mathcal{H} =
    $$
    $$
        =   \int\limits_{(-{2  \over \tau},{2 \over \tau}]}
        F_{\tau}(x)  \, d\sigma_{f}(x) + \int\limits_{(-{\infty},-{2  \over \tau}]}
        F_{\tau}(x)  \, d\sigma_{f}(x) + \int\limits_{({2  \over \tau},{\infty})}
        F_{\tau}(x)  \, d\sigma_{f}(x) \leq
    $$
    $$
        \leq     \int\limits_{(-{2  \over \tau},{2 \over \tau}]}
      \, d\sigma_{f}(x) +\frac{4}{\tau^2}\left( \ \int\limits_{(-{\infty},-{2  \over \tau}]}
     x^{-2}  \, d\sigma_{f}(x) + \int\limits_{({2  \over \tau},{\infty})}
       x^{-2} \, d\sigma_{f}(x) \right).
    $$
        Пусть $G(x)=\sigma_{f} (-\infty,\ x]$,\ $G(x)$ --- монотонная функция.
    Поскольку $\sigma_{f}$ является конечной борелевской мерой, то $G(x)$
    будет функцией ограниченной вариации на любом интервале вещественной
    прямой. Таким образом, наряду с интегралом Лебега--Стилтьеса можно
    рассмотреть интеграл Римана--Стилтьеса по функции $G(x)$. Если $h(x)$
    непрерывная на $[a,b]$ функция, то для нее существует интеграл
    Римана--Стилтьеса, и его значение совпадает
    с интегралом Лебега--Стилтьеса:
    $$
    \int\limits_{[a,b]} h(x)dG(x) = \int\limits_{[a,b]} h(x)d\sigma_{f}(x).
    $$
    Поэтому, если $h(x)$ непрерывна и имеет ограниченную вариацию на $[a,b]$, то
    $$
    \int\limits_{[a,b]} h(x)dG(x) =
    h(b)G(b) - h(a)G(a) - \int\limits_{[a,b]} G(x)dh(x).
    $$
    Это --- аналог формулы интегрирования по частям для интеграла
    Римана--Стилтьеса \ (см.~\cite[теорема~6.30]{Rud}; а также~\cite[гл.~II, \S6, теорема~11]{Sh}). Поскольку функция~$G$ непрерывна справа, то
    последняя формула справедлива как для конечных
    полуинтервалов~${(a,b]},$ так и бесконечных ${(-\infty,b]}$ и ${(a,+\infty)}.$
    Получаем:
    $$
     \|P_{t,s}f\|_\mathcal{H}^2 \leq
    \int\limits_{(-{2  \over \tau},{2 \over \tau}]}
    \, dG(x) + \frac{4}{\tau^2}\left( \ \int\limits_{(-{\infty},-{2  \over \tau}]}
    x^{-2}  \, dG(x) + \int\limits_{({2  \over \tau},{\infty})}
    x^{-2}\, dG(x)\right) =
    $$
    $$
    =G\left({2\over \tau}\right) - G\left({-2\over \tau}\right)+\frac{4}{\tau^2}\left( x^{-2}G(x)\Big|_{-\infty}^{-2\over \tau}+2\int\limits_{-\infty}^{-2\over \tau}
    x^{-3}G(x)  \, dx \right)+
    $$
    $$
    +\frac{4}{\tau^2}
    \left( x^{-2}G(x)\Big|_{2\over \tau}^{\infty}+2\int\limits_{2\over \tau}^{\infty}
    x^{-3}G(x)  \, dx \right)=
    $$
    $$
    =G\left({2\over \tau}\right) - G\left({-2\over \tau}\right)+G\left({-2\over \tau}\right)-G\left({2\over \tau}\right)
     + \frac{8}{\tau^2} \int\limits_{2\over \tau}^{\infty}
    x^{-3}(G(x)-G(-x)) \, dx =
    $$
    $$
    =\frac{8}{\tau^2} \int\limits_{2\over \tau}^{\infty}
    x^{-3}\sigma_{f}(-x,x] \, dx.
    $$
    Последнее равенство следует из замены $y=-x$ в интеграле $\int\limits_{-\infty}^{-2\over \tau}
    x^{-3}G(x)  \, dx$ и с учетом того, что $G(x)-G(-x) =\sigma_{f}(-x,x]$.
    Лемма~\ref{lmUpperEst} доказана.
\end{proof}

\begin{lemma}\label{lmLowerEst}
Пусть ${\varepsilon\in(0,\pi)};$ тогда для всех ${t>s}$ справедливо
неравенство
$$
\sigma_{f-f^*}\left(-{2\varepsilon \over t-s}, {2\varepsilon \over
t-s}\right]\leq
\frac{\varepsilon^2}{\sin^2\varepsilon}\|P_{t,s}f-f^*\|_\mathcal{H}^2.
$$
\end{lemma}
\begin{proof}[Доказательство леммы~\ref{lmLowerEst}] Снова, не ограничивая общности, считаем, что ${f^*=0}.$
Положим ${\tau=t-s}$ и воспользуемся представлением (1):
    $$
    \|P_{t,s}f\|_\mathcal{H}^2
    = \int \limits_{-\infty}^{+\infty} F_{\tau}(x) d\sigma_{f}(x)
    \geq \int \limits_{\left(-{2\varepsilon \over \tau}, {2\varepsilon \over \tau}\right]} F_{\tau}(x) d\sigma_{f}(x) \geq
    $$
    $$
    \geq \min\limits_{0<|x| \leq {2\varepsilon \over \tau}}
    \left({\sin{\tau x \over 2} \over {\tau x\over 2}}\right)^2
    \sigma_{f}\left(-{2\varepsilon \over \tau}, {2\varepsilon \over \tau}\right]
    =
    $$
    $$
    =\min\limits_{0<|y| \leq {\varepsilon}}
    \left({\sin{y} \over {y}}\right)^2
    \sigma_{f}\left(-{2\varepsilon \over \tau}, {2\varepsilon \over \tau}\right] =
    \frac{\sin^2\varepsilon}{\varepsilon^2}
    \sigma_{f}\left(-{2\varepsilon\over \tau}, {2\varepsilon \over \tau}\right].
    $$
    Лемма~\ref{lmLowerEst} доказана.
\end{proof}

%%%%%%%%%%%%%%%%%%%%%%%%%%%%%%%%%%%%%%%%%%%%%%%%%%%%%%%%%%%%%%%%%%%%%%%%%%%%%%%%%%%%%%%%%%%%%%%%%%%%%%%%%%%%%%

{\bf\thesection.2. Критерий степенной сходимости с показателем $\alpha<2$.} Для
${\alpha\in[0,2]}$ положим
$$
\rho(\alpha)=\inf\limits_{x>0}\frac{x^{2-\alpha}}{\sin^2x}.
$$
Нетрудно убедиться, что при ${\alpha=0}$ точная нижняя грань
достигается при ${x\to0+},$ и поэтому ${\rho(0)=1}.$ При
${\alpha=2}$ инфимум достигается в точках ${x_n=\pi/2+\pi n,
n\geq0}$ и также равен единице: ${\rho(2)=1}.$ Для
${\alpha\in(0,2)}$ инфимум достигается на первом положительном
решении уравнения ${\tg x=\frac{2x}{2-\alpha}}.$ Этот корень не
превосходит ${\pi/2},$ поэтому
${\rho(\alpha)\leq\left(\frac{\pi}{2}\right)^{2-\alpha}}$ при всех
${\alpha\in[0,2]}$. Эта оценка является хорошим приближением
$\rho(\alpha)$ при $\alpha$ близких к 2. Можно также легко
проверить, что ${1\leq\rho(\alpha)\leq\sin^{-2}(1)};$ при этом
значение ${\sin^{-2}(1)}$ достигается при ${\alpha=2(1-\ctg(1)).}$

%%%%% Критерий степенной сходимости в теореме фон Неймана на одномерных подпространствах%%%%%

\begin{theorem}\label{Th1}
    Пусть ${\alpha \in [0,2)};$ зафиксируем ${f \in \mathcal{H}}$. Тогда:

    1. Если спектральная мера $\sigma_{f-f^*}$ имеет степенную особенность в
    нуле, т.е. если для некоторой положительной константы $A$ при всех $\delta>0$ выполняется неравенство
    $$
    \sigma_{f-f^*}(-\delta,\delta]\leq A\delta^\alpha,
    $$
    то скорость сходимости эргодических средних $P_{t,s}f$ --- степенная с
    тем же показателем степени, т.е. при всех $t > s$
    $$
    \|P_{t,s}f-f^{*}\|^2_{\mathcal{H}} \leq B(t-s)^{-\alpha},
    $$
    где можно положить
    $B = \frac{2^{\alpha+1}}{2-\alpha}A.$

    2. Если скорость сходимости эргодических средних $P_{t,s}f$ ---
    степенная, т.е. если для некоторой положительной константы $B$ при
    всех $t > s$ выполняется неравенство
    $$
    \|P_{t,s}f-f^*\|_\mathcal{H}^2\leq B(t-s)^{-\alpha},
    $$
    то спектральная мера $\sigma_{f-f^*}$ имеет степенную особенность в
    нуле (с тем же показателем степени), т.е. для всех ${\delta>0}$
    $$
    \sigma _{f-f^*}(-\delta,\delta]\leq A\delta ^\alpha, \hbox { где }
    A = \frac{\rho(\alpha)}{2^\alpha} B\left(\le \frac{\pi^{2-\alpha}}4 B\right).
    $$
\end{theorem}

\begin{proof}[Доказательство теоремы~\ref{Th1}]
    Без ограничения общности считаем, что ${f^*=0}.$ Полагая ${\tau=t-s},$ по лемме~\ref{lmUpperEst} получаем:
    $$
 \|P_{t,s}f\|^2_{\mathcal{H}}
    \leq \frac{8}{\tau^2}\int \limits_{2\over \tau}^{\infty} x^{-3}\sigma_{f}(-x,x]dx=
    $$
    $$
    =\frac{8A}{\tau)^2}\int \limits_{2\over \tau}^{\infty} x^{-3}x^{\alpha}dx=
-\frac{8A}{(-2+\alpha)\tau^2}\left(\frac{2}{\tau}\right)^{-2+\alpha}=
\frac{2^{\alpha+1}}{2-\alpha}A\tau^{-\alpha}.
    $$

    Теперь докажем вторую часть теоремы. По лемме~\ref{lmLowerEst}
    %и учитывая степенную скорость убывания эргодических средних,
    при каждом $\varepsilon\in (0,\pi)$ для всех $\tau>0$ получаем:
    $$
    \sigma_{f}\left(-\frac{2\varepsilon}{\tau}, \frac{2\varepsilon}{\tau}\right] \leq
    \frac{\varepsilon^2}{\sin^2\varepsilon}B\tau^{-\alpha}.
    $$
Представляя произвольное ${\delta>0}$ в виде
${\delta=\frac{2\varepsilon}{\tau}}$ при ${\varepsilon\in(0,\pi)}$ и
$\tau>0,$ получаем:
$$
\sigma_{f}(-\delta,\delta]
\leq\frac{\varepsilon^{2-\alpha}}{\sin^2\varepsilon}\frac{B}{2^\alpha}\delta^{\alpha}.
$$
Минимизация константы в правой части по ${\varepsilon\in(0,\pi)}$
приводит к требуемому неравенству, что и завершает доказательство
теоремы~\ref{Th1}.
\end{proof}

Как будет показано далее (предложение 4), константа
$\rho(\alpha)/{2^\alpha}$ во второй части теоремы~\ref{Th1}
является точной (в~\cite{JK} соответствующая константа была чуть
хуже, и имела вид $(\frac{\pi}{2})^{2-\alpha}/{2}^\alpha).$

Разумеется, аналоги спектрального критерия степенной скорости сходимости теоремы 1 могут быть получены и для более широкого диапазона скоростей. Следующий диапазон был предложен В.Ф. Гапошкиным в~\cite{Gap}.

\begin{remark} Пусть $\alpha\in[0,2),$ и функция $\varphi (u)$ ---
    слабо колеблющаяся на $[1,\infty),$ т.е. для любого $\delta >0$
    функция $\varphi (u)u^\delta$ монотонно возрастает, а функция
    $\varphi (u)u^{-\delta }$ монотонно убывает. Тогда аналог утверждения теоремы 1 может быть получен и для всех скоростей сходимости вида $t^{-\alpha}\varphi(t)$ (утверждение теоремы 1 соответствует случаю $\varphi(t)\equiv 1$; при $\alpha=0$ и $\varphi(t)=\ln^\beta t,$ $\beta \ge 0$ получаем аналог этого утверждения для логарифмических скоростей; рассматриваемый диапазон скоростей включает и все скорости вида $t^{-\alpha}\ln^\beta t$ для всех $\alpha\in(0,2)$ при всех $\beta$).
\end{remark}

Доказательства аналогов теоремы 1 для указанных выше скоростей могут
быть получены так же, как и утверждения теоремы 1, конкретизацией
оценок лемм 1 и 2 для каждой из этих скоростей; поэтому мы и
сформулировали эти утверждения в виде двух отдельных лемм (вынеся их
из доказательства теоремы 1). Нас же в этой работе интересуют
исключительно степенные скорости.

%%%%%%%%%%%%%%%%%%%%%%%%%%%%%%%%%%%%%%%%%%%%%%%%%%%%%%%%%%%%%%%%

{\bf\thesection.3. Степенная сходимость с показателем ${\alpha\geq2}$.}
Для исчерпывающего
решения рассматриваемого в теореме~\ref{Th1} вопроса о критерии
наличия степенной скорости сходимости остается разобрать случай
$\alpha\geq 2$.

Оказывается, степенной скорости сходимости с показателем
${\alpha>2}$ не бывает (за исключением вырожденного случая $f=f^*$).
Доказательство этого факта, данное для эргодической теоремы фон
Неймана для полупотоков в замечании~3 в~\cite{JK} (полученное там
переносом соответствующего результата В.Ф. Гапошкина для дискретного
времени --- следствия 5 в~\cite{Ga75}
---  на время непрерывное) --- может быть перенесено и на наш случай почти
дословно. В замечании~\ref{rm2} раздела~4 мы дадим независимое совсем
простое доказательство того, что уже скоростей сходимости $o(t-s)^{-2}$
при $t-s\rightarrow \infty$ не бывает, т.е. что скорость сходимости
$O(t-s)^{-2}$ является максимально возможной.
Задача нахождения критерия наличия
такой максимально возможной (степенной с показателем ${\alpha=2}$)
скорости сходимости будет решена далее в разделе 4 теоремой~\ref{Th2},
и потребует корректировки подхода теоремы~\ref{Th1}.

Посмотрим, что можно получить при ${\alpha=2}$ старым подходом.
Аналог утверждения~2
теоремы~\ref{Th1} справедлив и в этом случае: доказательство
проходит без изменений. А вот аналог утверждения~1 этой теоремы не
имеет места, ни с какими константами: условие
$\sigma_{f-{f}^*}(-\delta, \delta] = O(\delta^2)$ при
$\delta\rightarrow 0$ не является, вообще говоря, достаточным для
выполнения соотношения $\|P_{t,s}f-f^*\|_\mathcal{H}^2= O(t-s)^{-2}$
при $t-s\rightarrow\infty$ (см. замечание 2 в~\cite{JK}). Тем не
менее, справедлив следующий ослабленный вариант этого соотношения.

\begin{propos}\label{pr1}
    Пусть $f\in\mathcal{H}$ и ${\alpha=2}.$ Если для некоторой
    положительной константы $A$ при всех $\delta>0$ выполняется неравенство
    $$\sigma_{f-f^*}(-\delta,\delta]\leq A_f\delta^2,$$
    то при ${t\geq s+2}$
    $$
    \|P_{t,s}f-f^{*}\|^2_{\mathcal{H}} \leq B(t-s)^{-2}\ln (t-s),
    $$
    где можно положить $B= 8A+\frac4{\ln2}\|f-f^*\|_{\mathcal{H}}^2.$
\end{propos}

\begin{proof}[Доказательство предложения~\ref{pr1}] Считаем ${f^*=0},$
и пусть ${\tau=t-s}.$  По лемме~\ref{lmUpperEst} при
${0<\frac{2}{\tau}\leq1}$ получаем:
    $$
     \|P_{t,s}f\|^2_{\mathcal{H}}
    \leq  \frac{8}{\tau^2}\int \limits_{2\over \tau}^{\infty} x^{-3}\sigma_{f}(-x,x]dx
    \leq \frac{8}{\tau^2}\left( \ \int \limits_{2\over \tau}^{1} x^{-3}\sigma_{f}(-x,x]dx
    + \int \limits_{1}^{\infty} x^{-3}\sigma_{f}(-x,x]dx\right)\leq
    $$
    $$
    \leq \frac{8}{\tau^2}\left(A\int \limits_{2\over \tau}^{1} x^{-1}dx
    + \int \limits_{1}^{\infty} x^{-3}\|f\|^2_{\mathcal{H}}dx\right)
    =\frac{8A\ln({\tau})}{\tau^2}- \frac{8A\ln 2}{\tau^2} +\frac{4\|f\|_{\mathcal{H}}^2}{\tau^2}
    \leq \frac{8A\ln(\tau)}{\tau^2}+\frac{4\|f\|_{\mathcal{H}}^2}{\tau^2}.
    $$
\end{proof}

Рассмотрим теперь случай ${\alpha>2}$.
Следующее предложение~\ref{pr2} в качестве аналога
утверждения~1 теоремы~\ref{Th1} для случая ${\alpha>2}$ дает
достаточный признак степенной скорости сходимости с максимально
возможным показателем степени, равным~2.

\begin{propos}\label{pr2}
    Пусть $f\in\mathcal{H}$ и ${\alpha>2}.$ Если для некоторой
    положительной константы $A$ при всех $\delta>0$ выполняется неравенство
    $$
    \sigma_{f-f^*}(-\delta,\delta]\leq A_f\delta^\alpha,
    $$
    то
    %скорость сходимости эргодических средних $P_{t,s}f$ --- максимальная, т.е.
    при ${t\geq s+2}$
    $$
    \|P_{t,s}f-f^{*}\|^2_{\mathcal{H}} \leq B(t-s)^{-2},
    $$
    где можно положить $B=\frac{8}{\alpha-2}A+4\|f-f^*\|^2_{\mathcal{H}}.$
\end{propos}

\begin{proof}[Доказательство предложения~\ref{pr2}] Снова считаем, что ${f^*=0},$
и ${\tau=t-s}.$ По лемме~\ref{lmUpperEst} при
${0<\frac{2}{\tau}\leq1}$ получаем:
$$
 \|P_{t,s}f\|^2_{\mathcal{H}}
\leq \frac{8}{\tau^2}\int \limits_{2\over \tau}^{\infty}
x^{-3}\sigma_{f}(-x,x]dx\leq \frac{8}{\tau^2}\left( \ \int
\limits_{2\over \tau}^{1} x^{-3}\sigma_{f}(-x,x]dx + \int
\limits_{1}^{\infty} x^{-3}\sigma_{f}(-x,x]dx\right)\leq
$$
$$
\leq \frac{8}{\tau^2}\left(A\int \limits_{2\over \tau}^{1}
x^{-3+\alpha}dx + \int \limits_{1}^{\infty}
x^{-3}\|f\|^2_{\mathcal{H}}dx\right) =
\frac{8A}{(\alpha-2)\tau^2}\left(1-\frac{2^{\alpha-2}}{\tau^{\alpha-2}}\right)
+\frac{4\|f\|_{\mathcal{H}}^2}{\tau^2} =
$$
$$
=\frac{2^{1+\alpha}A}{(2-\alpha)\tau^{\alpha}}
+\left(\frac{8A}{\alpha-2}+4\|f\|_{\mathcal{H}}^2\right)\frac{1}{\tau^2}
\leq\left(\frac{8A}{\alpha-2}+4\|f\|^2_{\mathcal{H}}\right)\frac{1}{\tau^2}.
$$
\end{proof}

\section{Общий многомерный случай}

{\bf\thesection.1. Степенные равномерные сходимости на
подпространствах.} Следующая теорема является следствием
(естественной  переформулировкой на общий многомерный случай)
теоремы~\ref{Th1}.

\begin{theorem}\label{ThMain}
Пусть $\alpha \in [0,2), \mathcal{H}$ --- гильбертово пространство,
$\mathcal{X} \subseteq \mathcal{H}$ --- его векторное подпространство
со своей нормой  $\|\cdot\|_{\mathcal{X}}$. Тогда:

    1. Если существует положительная константа $A,$ такая, что
    для всех $f \in \mathcal{X}$ при всех $\delta>0$ выполняется неравенство
    $$
    \sigma_{f-f^*}(-\delta,\delta]\leq A \|f\|^2_{\mathcal{X}}\delta^\alpha,
    $$
    то имеет место степенная равномерная сходимость на пространстве $\mathcal{X}$ в теореме фон Неймана: при всех $t>s$
    $$
    \|P_{t,s} - P\|^2_{\mathcal{X} \to \mathcal{H}} \leq B(t-s)^{-\alpha},
    $$
    где можно положить $B = \frac{2^{\alpha+1}}{2-\alpha}A.$

    2. Если имеет место степенная равномерная сходимость на пространстве
    $\mathcal{X}$ в теореме фон Неймана,
    т.е. для некоторой положительной константы $B$ при всех $t>s$ выполняется неравенство
$$
\|P_{t,s} - P\|^2_{\mathcal{X} \to \mathcal{H}} \leq B(t-s)^{-\alpha},
$$
то для всех $f\in \mathcal{X}$ спектральная мера $\sigma_{f-f^*}$ имеет степенную особенность в
нуле (с тем же показателем степени), т.е. при всех ${\delta>0}$
$$
\sigma_{f-f^*}(-\delta,\delta]\leq A \|f\|^2_{\mathcal{X}}\delta^\alpha, \hbox { где }
A = \frac{\rho(\alpha)}{2^\alpha} B\left(\le \frac{\pi^{2-\alpha}}4 B\right).
$$
\end{theorem}

\begin{proof}[Доказательство теоремы~\ref{ThMain}]
Из первой части теоремы~\ref{Th1} следует, что
$$
\|(P_{t,s} - P)f\|^2_{\mathcal{H}} \leq B(t-s)^{-\alpha}\|f\|^2_{\mathcal{X}},
$$
 где
${B = \frac{2^{\alpha+1}}{2-\alpha}A.}$ Поэтому для всех ${t > s}$
    $$
    \|P_{t,s} - P\|^2_{\mathcal{X} \to \mathcal{H}}=
     \sup\limits_{f \in \mathcal{X} : f \neq 0} \frac{\|(P_{t,s} - P)f\|^2_\mathcal{H}}{\|f\|^2_{\mathcal{X}}} \leq B (t-s)^{-\alpha}.
     $$

Докажем вторую часть теоремы. Из равномерной сходимости
эргодических средних получаем, что для всех ${f\in\mathcal{X}}$
при всех ${t>s}$
        $$
        \frac{\|(P_{t,s} - P)f\|^2_\mathcal{H}}{\|f\|^2_{\mathcal{X}}} \leq B (t-s)^{-\alpha}.
        $$
        Тогда из второй части теоремы~\ref{Th1} следует, что для всех ${\delta>0}$
        $$
        \sigma_{f-f^*}(-\delta,\delta]\leq A\delta^\alpha \|f\|^2_{\mathcal{X}}, \hbox { где }
        A = \frac{\rho(\alpha)}{2^\alpha}B,
        $$
    что и требовалось доказать.
\end{proof}

В теореме~\ref{Th2+} раздела 4 ниже мы дадим критерий наличия
степенной с показателем ${\alpha=2}$ (максимально возможным --- по
замечанию~\ref{rm2} того же раздела 4) равномерной сходимости на
подпространствах. А пока переформулируем на многомерный случай
утверждения предложений~\ref{pr1} и \ref{pr2} предыдущего раздела.

\begin{propos}\label{pr3}
    Пусть $\alpha \geq 2 , \mathcal{H}$ --- гильбертово пространство,
    $\mathcal{X} \subseteq \mathcal{H}$ --- его векторное подпространство
    со своей нормой  $\|\cdot\|_{\mathcal{X}},$ которое непрерывно в
    него вложено, т.е. без ограничения общности считаем
     $\|\cdot\|_\mathcal{H} \leq \|\cdot\|_\mathcal{X}$.

     Если существует положительная константа $A,$ такая, что
     для всех $ f \in \mathcal{X}$  при всех $\delta>0$ выполняется неравенство
     $$
     \sigma_{f-f^*}(-\delta,\delta]\leq A \|f\|^2_{\mathcal{X}}\delta^2,
     $$
    то имеет место равномерная сходимость на пространстве $\mathcal{X}$ в теореме фон Неймана:

    1. В случае $\alpha=2$ при ${t\geq s+2}$
    $$
    \|P_{t,s} - P\|^2_{\mathcal{X} \to \mathcal{H}} \leq B\frac{\ln(t-s)}{(t-s)^2},
    $$
    где можно положить $B = 8A+\frac4{\ln2}.$

    2. В случае $\alpha>2$ при ${t\geq s+2}$
    $$
    \|P_{t,s} - P\|^2_{\mathcal{X} \to \mathcal{H}} \leq B(t-s)^{-2},
    $$
    где можно положить
    $B = \frac{8}{\alpha-2}A+4.$
\end{propos}

%%%%%%%%%%%%%%%%%%%%%%%%%%%%%%%%%%%%%%%%%%%%%%%%%%%%%%%%%

{\bf\thesection.2. Пространства ${\mathcal{X}_\alpha}$ и
$\mathcal{Y}$.} Для любого ${\alpha>0}$ обозначим через
$\mathcal{X}_\alpha$ множество всех~${f\in\mathcal{H}},$ у которых
спектральная мера~${\sigma_f}$ имеет степенную с
показателем~${\alpha>0}$ особенность в нуле, т.е.
$$
\mathcal{X}_\alpha = \{f\in \mathcal{H}|\ \exists A>0\ \forall
\delta>0\ \ \sigma_{f}(-\delta,\delta] \leq A\delta^\alpha\}.
$$

Покажем, что локальная степенная оценка на спектральную
меру эквивалентна глобальной степенной оценке.

\begin{lemma}\label{lmLoc}
     Пусть ${\alpha>0}$ и ${f\in \mathcal{H}}.$ Тогда следующие условия эквивалентны:

     (1) Существует константа ${A>0}$ такая, что  ${\sigma_{f}(-\delta,\delta]\leq A\delta^\alpha}$ для всех ${\delta>0};$

     (2) Существуют число ${r>0}$ и константа ${A_r>0}$ такие, что
     ${\sigma_{f}(-\delta,\delta]\leq A_r\delta^\alpha}$ для всех ${\delta \in (0,r)}.$
\end{lemma}

\begin{proof}[Доказательство леммы~\ref{lmLoc}] Переход ${(1)\Rightarrow(2)}$ очевиден; покажем
${(2)\Rightarrow(1)}.$ Используя равенство
${\sigma_{f}(\mathbb{R})=\|f\|^2_\mathcal{H}},$  для всех ${\delta
\geq r}$ получаем
$$
\sigma_{f}(-\delta,\delta]\leq\sigma_{f}(\mathbb{R})=\|f\|^2_{\mathcal{H}}\leq
\frac{\|f\|^2_\mathcal{H}}{r^\alpha}\delta^\alpha.
$$
Следовательно, для всех ${\delta>0}$
$$\sigma_{f}(-\delta,\delta]\leq A\delta^\alpha,$$
    где $A=\max\left\{A_r,\frac{\|f\|^2_\mathcal{H}}{r^\alpha}\right\}.$
\end{proof}

Нетрудно убедиться, что множества ${\mathcal{X}_\alpha}$ образуют
убывающую по включению цепь, а именно
${\mathcal{X}_{\alpha_1}\subseteq\mathcal{X}_{\alpha_2}}$ при
${\alpha_1>\alpha_2}.$ Действительно, пусть
${f\in\mathcal{X}_{\alpha_1}};$ тогда по лемме~\ref{lmLoc}
достаточно получить локальную степенную оценку.  Для всех
${\delta\in(0,1)}$ будет ${\sigma_f(-\delta,\delta]\leq
A\delta^{\alpha_1}\leq A\delta^{\alpha_2}};$ следовательно,
${f\in\mathcal{X}_{\alpha_2}}.$

\begin{propos}\label{pr4}
    ${\mathcal{X}_\alpha}$ для любого ${\alpha>0}$
     образует векторное подпространство в~$\mathcal{H}$,
     ортогональное подпространству неподвижных векторов группы
    ${\{U^t\}_{t\in \R}};$ в нем можно ввести норму
    $$
    ||f||_{\mathcal{X}_\alpha}^2 = \inf \{A | \forall  \delta >0   \ \sigma_{f}(-\delta,\delta] \leq A\delta^\alpha \}.
    $$
\end{propos}

\begin{proof}[Доказательство предложения~\ref{pr4}]
    Проверим справедливость всех трех аксиом нормы.

    1) Пусть  $f \in \mathcal{X}_\alpha$ такова, что $\|f\|^2_{\mathcal{X}_\alpha}=0$, т.е. $\sigma_{f}(-\delta,\delta]=0 $ для всех $\delta>0$.
    Покажем, что $f=0$. Для этого введем последовательность $\sigma_n = \sigma_{f}(-n,n].$ С одной стороны, $\lim\limits_{n\rightarrow\infty}\sigma_n=0 $.
    С другой стороны, $\lim\limits_{n\rightarrow\infty}\sigma_{f}(-n,n]=\sigma_{f}(\R)=(f,f)_\mathcal{H}.$
    Отсюда следует, что ${f=0}.$

    2) Покажем, что $\|\lambda f\|^2_{\mathcal{X}_\alpha}=|\lambda|^2\|f\|^2_{\mathcal{X}_\alpha}$ для всех ${f \in \mathcal{X}_\alpha, \lambda \in \mathbb{C}}.$
    При всех $\delta>0$
    $$\sigma_{\lambda f}(-\delta,\delta]=
    \int \limits_{(-\delta,\delta ]}d(E(x)\lambda f,\lambda f)_\mathcal{H}=
    |\lambda|^2  \int \limits_{(-\delta,\delta ]}d(E(x)f,f)_\mathcal{H}=
    |\lambda|^2\sigma_{f}(-\delta,\delta].$$
    Таким образом, мы получаем, что если ${\sigma_{f}(-\delta,\delta] \leq A_f\delta^\alpha},$ то $\sigma_{\lambda f}(-\delta,\delta] \leq |\lambda|^2A_f\delta^\alpha.$
    Беря инфимум по таким $A_f,$ мы получаем, что $\|\lambda f\|^2_{\mathcal{X}_\alpha} \leq|\lambda|^2\|f\|^2_{\mathcal{X}_\alpha}.$

    Без труда проверяется, что из этого неравенства автоматически следует равенство, а именно:
    $$\|f\|_{\mathcal{X}_\alpha} = \|\lambda \frac{1}{\lambda}f\|_{\mathcal{X}_\alpha} \leq |\lambda| \| \frac{1}{\lambda}f\|_{\mathcal{X}_\alpha} \leq |\lambda| |\frac{1}{\lambda}|\|f\|_{\mathcal{X}_\alpha}.$$

    3) Пусть $f,g \in \mathcal{X}_\alpha ;$ тогда $\sigma_{f}(-\delta,\delta] \leq \|f\|^2_{\mathcal{X}_\alpha} \delta^\alpha$ и
    $\sigma_{g}(-\delta,\delta] \leq \|g\|^2_{\mathcal{X}_\alpha}\delta^\alpha$
    при всех ${\delta>0}.$ Используя неравенство ${|\sigma_{f,g}(-\delta,\delta]|^2\leq
    \sigma_{f}(-\delta,\delta]\sigma_{g}(-\delta,\delta]},$
    справедливое для всех~${f,g\in\mathcal{H}}$ (см.,
    например,~\cite[5.5]{Gl03}), получаем:
    $$
    \sigma_{f+g}(-\delta,\delta]=\sigma_{f}(-\delta,\delta] + \sigma_{f, g}(-\delta,\delta]+\sigma_{g, f}(-\delta,\delta]+\sigma_{g}(-\delta,\delta] \leq
    $$
    $$
    \leq\sigma_{f}(-\delta,\delta]  +\sqrt{\sigma_{f}(-\delta,\delta]\sigma_{g}(-\delta,\delta]}
    +\sqrt{\sigma_{g}(-\delta,\delta]\sigma_{f}(-\delta,\delta]} +\sigma_{g}(-\delta,\delta]\leq
    $$
    $$
    \leq\|f\|^2_{\mathcal{X}_\alpha}\delta^\alpha+ 2\|f\|_{\mathcal{X}_\alpha}\|g\|_{\mathcal{X}_\alpha}\delta^\alpha
    +\|g\|^2_{\mathcal{X}_\alpha}\delta^\alpha = (\|f\|_{\mathcal{X}_\alpha}+\|g\|_{\mathcal{X}_\alpha})^2\delta^\alpha.
    $$
    Следовательно, $f+g \in \mathcal{X}_\alpha,$ и
    $$
    ||f+g||^2_{\mathcal{X}_\alpha} \leq(||f||_{\mathcal{X}_\alpha}+\|g\|_{\mathcal{X}_\alpha})^2 ,
    $$
    что и требовалось.

Нетрудно видеть, что для любого вектора ${f\in\mathcal{X}_\alpha}$
значение его спектральной меры в нуле
${\sigma_f(\{0\})=\|f^*\|^2_\mathcal{H}=0}.$ Поэтому ${Pf=f^*=0}.$
\end{proof}

Определяя ${\mathcal{X}_0}$ аналогично ${\mathcal{X}_\alpha}$, т.е.
при ${\alpha=0},$ нетрудно видеть, что тогда
$\mathcal{X}_0=\mathcal{H}$ и
${\|\cdot\|_{\mathcal{X}_0}=\|\cdot\|_\mathcal{H}}.$ Рассмотрим еще
одно подпространство~$\mathcal{Y},$ состоящее из векторов образа
${\mathcal{R}(\mathbf{B})}$ генератора $\mathbf{B}$ с нормой,
определяемой равенством
$$
\|f\|_{\mathcal{Y}}=\sup_{t>0}\left\|\int_0^tU^\tau
f\,d\tau\right\|_\mathcal{H},\ \ f\in\mathcal{R}(\mathbf{B}).
$$
Хорошо известно~\cite[лемма VIII.1.7]{DSh1}, что
${\|f\|_{\mathcal{Y}}<\infty}$ для любого вектора
${f\in\mathcal{R}(\mathbf{B})},$ а именно для ${g\in
\mathrm{Dom}(\mathbf{B})}$ и всех ${t>0}$
$$
\left\|\int_0^tU^\tau
\mathbf{B}g\,d\tau\right\|_\mathcal{H}=\|(U^t-I)g\|_\mathcal{H}\leq2\|g\|_\mathcal{H}.
$$
Для некоторых полугрупп верно и обратное утверждение: если
${\|f\|_{\mathcal{Y}}<\infty}$, то ${f\in\mathcal{R}(\mathbf{B})}.$
В частности, для унитарной группы это будет справедливо ---
см.~\cite[теорема 2.6]{KL84}, где доказано это свойство для
двойственных (состоящих из банахово сопряженных операторов) сильно
непрерывных полугрупп. Нетрудно также проверить справедливость всех
свойств нормы для ${\|\cdot\|_{\mathcal{Y}}}.$

%%%%%%%%%%%%%%%%%%%%%%%%%%%%%%%%%%%%%%%%%%%%%%%%%%%%%%%%%%%%%%%%%%%%%%%%%
{\bf\thesection.3. Характеризация подпространств со степенной
равномерной сходимостью.} С помощью
подпространств~${\mathcal{X}_\alpha}, {\alpha\in[0,2)}$ и
$\mathcal{Y}$ можно описать все вложенные в $\mathcal{H}$
нормированные пространства, на которых имеется равномерная степенная
сходимость в теореме фон Неймана. Напомним, что (замечание~\ref{rm2}
раздела 4) такой скорости сходимости с показателем степени больше 2
не бывает.

\begin{theorem}\label{ThCharacterizationOfSpaces}
Пусть ${\alpha \in [0,2]},$ $\mathcal{X} \subseteq \mathcal{H}$ ---
его векторное подпространство со своей
нормой~$\|\cdot\|_{\mathcal{X}}.$ На пространстве $\mathcal{X}$
будет равномерная степенная с показателем $\alpha$ сходимость в
теореме фон Неймана тогда и только тогда, когда $I-P$ является
непрерывным  вложением $\mathcal{X}$ в ${\mathcal{X}_\alpha}$ в
случае ${\alpha\in[0,2)},$ и, соответственно, непрерывным вложением
$\mathcal{X}$ в $\mathcal{Y}$ при ${\alpha=2}.$
\end{theorem}

\begin{proof}[Доказательство теоремы~\ref{ThCharacterizationOfSpaces}]
Рассмотрим сначала случай ${\alpha\in[0,2)}.$
Пусть существует константа ${B>0}$ такая, что ${\|P_{t,s} -
P\|^2_{\mathcal{X} \to \mathcal{H}} \leq B(t-s)^{-\alpha}}$ для всех
${t>s}.$

Из доказательства второго утверждения теоремы~\ref{ThMain} следует,
что при ${\alpha\in[0,2)}$ (и при ${\alpha=2}$) для всех ${f\in
\mathcal{X}}$ спектральная мера вектора $f-f^*=(I-P)f$ имеет
степенную с показателем $\alpha$ особенность, а именно, для всех
${\delta>0}$
$$
\sigma_{(I-P)f}(-\delta,\delta]\leq \frac{\rho(\alpha)}{2^\alpha} B
\|f\|^2_{\mathcal{X}}\delta^\alpha.
$$
Следовательно, ${(I-P)f\in\mathcal{X}_\alpha}$ и
$\|(I-P)f\|^2_{\mathcal{X}_\alpha}\leq\frac{\rho(\alpha)}{2^\alpha}
B \|f\|^2_{\mathcal{X}}$ для всех $f\in \mathcal{X},$
что и требовалось доказать.

Докажем утверждение в обратную сторону. Пусть
${(I-P)\mathcal{X}\subseteq\mathcal{X}_\alpha}$ и существует
положительная константа $A$ такая, что
${\|I-P\|^2_{\mathcal{X}\to\mathcal{X}_\alpha}\leq A}.$ Тогда для
любого ${\delta>0}$
$$
\sigma_{(I-P)f}(-\delta,\delta]\leq\|(I-P)f\|^2_{\mathcal{X}_\alpha}\delta^\alpha\leq
A\|f\|^2_\mathcal{X}\delta^\alpha,
$$
и первое утверждение теоремы~\ref{ThMain} дает при всех $t>s$ требуемую оценку
$$
\|P_{t,s}-P\|^2_{\mathcal{X}\to\mathcal{H}}\leq\frac{2^{\alpha+1}}{2-\alpha}A(t-s)^{-\alpha}.
$$
Пусть теперь ${\alpha=2}$ и существует константа ${B>0}$ такая, что
для всех ${t>s}$ выполняется неравенство ${\|P_{t,s} -
P\|^2_{\mathcal{X} \to \mathcal{H}} \leq B(t-s)^{-2}}.$ Полагая
${s=0},$ получаем при всех $t>0$ оценку
$$
\left\|\int_0^tU^\tau (I-P)f\,d\tau\right\|^2_\mathcal{H}\leq
B\|f\|^2_\mathcal{X}.
$$
Следовательно, ${(I-P)f\in\mathcal{Y}}$ и
${\|(I-P)f\|^2_\mathcal{Y}\leq B\|f\|^2_\mathcal{X}},$ что и
требовалось доказать. В обратную сторону: пусть
${(I-P)\mathcal{X}\subseteq\mathcal{Y}}$ и существует положительная
константа $A$ такая, что ${\|I-P\|^2_{\mathcal{X}\to\mathcal{Y}}\leq
A}.$ Тогда при всех ${t>s}$
\begin{multline*}
\|P_{t,s}f-Pf\|^2_\mathcal{H}
=\frac{1}{(t-s)^2}\left\|\int_s^tU^\tau(I-P)fd\tau
\right\|^2_\mathcal{H}=\frac{1}{(t-s)^2}\left\|U^s\int_0^{t-s}U^\tau(I-P)fd\tau
\right\|^2_\mathcal{H}=\\
=\frac{1}{(t-s)^2}\left\|\int_0^{t-s}U^\tau(I-P)fd\tau
\right\|^2_\mathcal{H}\leq\frac{1}{(t-s)^2}\|(I-P)f\|^2_\mathcal{Y}
\leq A\|f\|^2_{\mathcal{X}}(t-s)^{-2}.
\end{multline*}
\end{proof}

\begin{remark}[М. Лин \cite{Lin74}] \label{rm0}
Равномерная сходимость в теореме фон Неймана на всем пространстве
$\mathcal{H}$ имеет место тогда и только тогда, когда образ
$\mathcal{R}(\mathbf{B})$ генератора $\mathbf{B}$ группы будет
замкнутым.
\end{remark}

При этом ${\mathcal{H}=\mathcal{R}(\mathbf{B})\oplus\{x: U^tx=x\
\forall t>0\}},$ сужение~$\mathbf{B}_1$ генератора $\mathbf{B}$ на
$\mathrm{Dom}\,(\mathbf{B})\cap\mathcal{R}(\mathbf{B})$ будет
обратимым, и ${\|\mathbf{B}_1^{-1}\|<\infty}.$ Поэтому (см.
доказательство в~\cite{Lin74}) в рассматриваемом случае скорость сходимости
будет максимально возможной. А именно,
для любого вектора ${f\in\mathcal{H}},$ поскольку
${(I-P)f\in\mathcal{R}(\mathbf{B})=\mathcal{R}(\mathbf{B}_1)},$
при всех ${t>s}$ будет
\begin{multline*}
\|P_{t,s}f-Pf\|_\mathcal{H}=\frac{1}{t-s}\left\|U^s\int_0^{t-s}U^\tau(I-P)fd\tau
\right\|_\mathcal{H}=\frac{1}{t-s}\left\|\int_0^{t-s}U^\tau(I-P)fd\tau
\right\|_\mathcal{H}=\\
=\frac{1}{t-s}\|(U^{t-s}-I)\mathbf{B}^{-1}(I-P)f\|_\mathcal{H}
\leq\frac{2}{t-s}\|\mathbf{B}_1^{-1}\|\|(I-P)f\|_\mathcal{H}
\leq\frac{2\|\mathbf{B}_1^{-1}\|\|f\|_\mathcal{H}}{t-s}.
\end{multline*}

Этот критерий можно переформулировать в терминах существования
спектрального пробела ("spectral gap"); см. также обсуждение для
общих полугрупп в~\cite[1.2.15]{Em07}.

\begin{remark}\label{rm0+}
Равномерная сходимость в теореме фон Неймана на всем пространстве
$\mathcal{H}$ имеет место тогда и только тогда, когда существует
${\gamma>0}$ такое, что во множестве
${(-\gamma,\gamma)\setminus\{0\}}$ нет точек спектра
оператора~$\mathbf{B}.$
\end{remark}

Действительно, если сходимость равномерная на всем пространстве, то
по предыдущему замечанию спектр генератора~$\mathbf{B}$ совпадает со
спектром оператора~${\mathbf{B}_1}$ в объединении с точкой 0 (если
$\{x: U^tx=x\ \forall t>0\}=\mathrm{ker}\,\mathbf{B}\neq\{0\}$).
Поскольку для оператора ${\mathbf{B}_1}$ точка 0 является регулярным
значением (ввиду того, что
$\mathcal{R}(\mathbf{B}_1)=\mathcal{R}(\mathbf{B})$ и
$\|\mathbf{B}^{-1}_1\|<\infty$), а множество регулярных значений для
замкнутого оператора открыто, то найдется проколотая окрестность
точки ноль, содержащая только регулярные значения генератора.

Обратно, если есть такое ${\gamma>0,}$ то, используя
представление~(1), получим для каждого вектора ${f\in\mathcal{H}}$
при всех ${t>s}$ оценку
\begin{align*}
\|P_{t,s}f-Pf\|^2_\mathcal{H}&=\frac{4}{(t-s)^2}\int_{\mathbb{R}\setminus\{0\}}\frac{\sin^2(\frac{(t-s)x}{2})}{x^2}\,d\sigma_f(x)=\\
&=\frac{4}{(t-s)^2}\int_{\mathbb{R}\setminus(-\gamma,\gamma)}\frac{\sin^2(\frac{(t-s)x}{2})}{x^2}\,d\sigma_f(x)
\leq\frac{4\|f\|^2_\mathcal{H}}{(t-s)^2\gamma^2},
\end{align*}
что и завершает доказательство замечания 2.

Поскольку генератор унитарной группы, порожденной потоком
сохраняющих меру автоморфизмов, имеет свою специфику (см.,
например,~\cite{Lem17}), нам здесь интересно описание потоков, для
которых имеется равномерная сходимость в теореме фон Неймана на всем
пространстве. В дискретном случае известно, что такая ситуация имеет
место только для периодического автоморфизма (см. обсуждение
в~\cite{GHT}). В следующем замечании содержится аналог этого
утверждения для непрерывного времени.

Пусть ${\mathbf{T}=\{T^t\}_{t\in\mathbb{R}}}$ --- поток, действующий
на пространстве Лебега ${(\Omega,\mu)}$ с неатомической
мерой~$\mu,$ и ${U_{\mathbf{T}}^tf=f\circ T^t}$ --- его группа
унитарных операторов Купмана в ${L_2(\Omega,\mu)}.$ Пусть
$\mathcal{P}(\omega)$ --- период точки~${\omega\in\Omega}$
относительно потока ${\mathbf{T}},$ т.е. такое минимальное число
$t>0,$ что ${T^t\omega=\omega};$ полагаем для непериодических точек
${\mathcal{P}(\omega)=\infty}$.

\begin{remark}\label{rm0++}
Равномерная сходимость в теореме фон Неймана для унитарной
группы~$\{{U_{\mathbf{T}}^t}\}_{t\in\mathbb{R}}$ на всем
пространстве~${L_2(\Omega,\mu)}$ имеет место тогда и только
тогда, когда ${\mathcal{P}\in L_\infty(\Omega,\mu)}.$
\end{remark}

Действительно, если имеется равномерная сходимость на всем
пространстве, то по замечанию~\ref{rm0+} спектр генератора группы
имеет гэп. Тогда $\mathcal{P}\in
L_\infty(\Omega,\mu)$, поскольку иначе, как показано в~\cite{G75}
(случай апериодического потока см. также в~\cite{N79}) спектр генератора
есть все $\mathbb{R}.$

Обратно, пусть ${\|\mathcal{P}\|_\infty<\infty}.$ Тогда
${f^*(\omega)=\frac{1}{\mathcal{P}(\omega)}\int_0^{\mathcal{P}(\omega)}f(T^\tau\omega)\,d\tau}$
для п.в. $\omega\in\Omega$;
%(по эргодической теореме Биркгофа для орбиты точки $\omega$
%с инвариантной эргодической одномерной мерой Лебега)
полагая $0<t-s=N\mathcal{P}(\omega)+r, \, 0\le r<\mathcal{P}(\omega), N\in\mathbb{N}\cup\{0\},$
для всех ${t>s}$ получаем
$$
\|P_{t,s}f-f^*\|_2=
\left\|\frac{1}{t-s}\int_0^{t-s}\!\!\!f(T^{\tau+s}\omega)\,d\tau-f^*(\omega)\right\|_2=
\left\|\frac{1}{t-s}\int_0^{t-s}\!\!\!f(T^{\tau}\omega)\,d\tau-f^*(\omega)\right\|_2=
$$
$$
\\
=\left\|\frac{1}{t-s}\sum_{k=0}^{N-1}\int_{k\mathcal{P}(\omega)}^{(k+1)\mathcal{P}(\omega)}\!\!f(T^\tau\omega)\,d\tau+
\frac{1}{t-s}\int_{N\mathcal{P}(\omega)}^{t-s}\!\!f(T^\tau\omega)\,d\tau-\frac{1}{\mathcal{P}(\omega)}\int_0^{\mathcal{P}(\omega)}\!\!f(T^\tau\omega)\,d\tau\right\|_2=
$$
$$
\\
=\left\|\left(\frac{N}{t-s}-\frac{1}{\mathcal{P}(\omega)}\right)\int_0^{\mathcal{P}(\omega)}f(T^\tau\omega)\,d\tau+
\frac{1}{t-s}\int_{0}^{r}f(T^\tau\omega)\,d\tau\right\|_2\leq
$$
$$\\
\leq\frac{2}{t-s}\left\|\int_0^{\mathcal{P}(\omega)}|f(T^\tau\omega)|\,d\tau\right\|_2\leq
\frac{2}{t-s}\left\|\int_0^{\|\mathcal{P}\|_\infty}|f(T^\tau\omega)|\,d\tau\right\|_2\leq\frac{2\|\mathcal{P}\|_\infty\|f\|_2}{t-s}.
$$

Если функция ${\mathcal{P}}$ не будет существенно ограниченной, а
лишь интегрируемой с некоторой степенью, то можно получить
равномерную сходимость с максимальной скоростью на соответствующих классических
подпространствах интегрируемых функций. А именно, справедливо следующее утверждение.
%В следующем замечании
%приводится точная формулировка этого утверждения.

\begin{remark}\label{rm0+++} Пусть ${\mathcal{P}\in
L_{2q}(\Omega,\mu)}$ для ${1\leq q<\infty}.$ Тогда на
пространстве~${L_{2p}(\Omega,\mu)},$ где
${\frac{1}{q}+\frac{1}{p}=1},$ будет иметь место равномерная
степенная сходимость с показателем ${\alpha=2}$ в теореме фон
Неймана для группы~$\{{U_{\mathbf{T}}^t}\}_{t\in\mathbb{R}}.$
\end{remark}

Действительно, полагая
${f^{**}(\omega)=\sup\limits_{t>0}\frac{1}{t}\int_0^t|f(T^\tau\omega)|\,d\tau}$
и применяя неравенство Гёльдера, из выкладок предыдущего замечания
получаем для всех ${t-s>0}$
$$
\|P_{t,s}f-f^*\|_2
\leq\frac{2}{t-s}\left\|\int_0^{\mathcal{P}(\omega)}|f(T^\tau\omega)|\,d\tau\right\|_2
\leq \frac{2}{t-s}\left\|\mathcal{P}(\omega)f^{**}(\omega)\right\|_2
\leq\frac{2\|\mathcal{P}\|_{2q}\|f^{**}\|_{2p}}{t-s}.
$$
При $q>1,p<\infty$ доказываемое утверждение следует отсюда по доминантному неравенству
${\|f^{**}\|_{2p}\leq\frac{2p}{2p-1}\|f\|_{2p}};$ при $q=1,p=\infty$
работает очевидное неравенство $\|f^{**}\|_\infty\leq \|f\|_\infty$.

%\$$
%\|P_{t,s}f-f^*\|_2\leq\frac{4p}{2p-1}\frac{\|\mathcal{P}\|_{2q}\|f\|_{2p}}{t-s}
%$$
%при всех ${t-s>0}.$

%%%%%%%%%%%%%%%%%%%%%%%%%%%%%%%%%%%%%%%%%%%%%%
{\bf\thesection.4. О неулучшаемости константы
$\frac{\rho(\alpha)}{2^\alpha}$ в оценках теорем~\ref{Th1} и~\ref{ThMain}.}
Еще раз воспользуемся
подпространствами ${\mathcal{X}_\alpha}$ из пункта 3.2 выше.
На гильбертовом пространстве
${\mathcal{H}=L_2(\mathbb{R})}$ рассмотрим группу унитарных
операторов умножения
$$
U^tf(x)=e^{ixt}\cdot f(x),\ f\in L_2(\mathbb{R}).
$$
Легко определяется инфинитезимальный генератор $\mathbf{B}$ этой
группы
--- это самосопряженный оператор умножения на $x$ с естественной
областью определения, т.е.
$$
\mathbf{B}f(x)=x\cdot f(x), \ f\in \mathrm{Dom}\,\mathbf{B}=\{g\in
L_2(\mathbb{R}):\ xg(x)\in L_2(\mathbb{R})\}.
$$
Для любого борелевского подмножества ${C\subseteq\mathbb{R}}$
спектральные проекторы определяются равенством
${E(C)f(x)=\chi_C(\mathbf{B})f(x)=\chi_{C}(x)\cdot f(x)}.$ Тогда
$$
\sigma_f(C)=(E(C)f,f)=\int_\mathbb{R}\chi_C(x)|f(x)|^2\,dx=\int_C|f(x)|^2\,dx,
$$
т.е. спектральная мера $\sigma_f$ абсолютно непрерывна и имеет
плотность $|f|^2.$ Пространства ${\mathcal{X}_\alpha}$ в этом случае
описываются следующим образом:
$$
\mathcal{X}_\alpha=\left\{f\in L_2(\mathbb{R}):\
\sup\limits_{\delta>0}\frac{1}{\delta^\alpha}\int_{-\delta}^\delta|f(x)|^2\,dx<\infty\right\}.
$$
Представителями этого пространства являются, например, функции вида
$f(x)=\chi_{(a,b)}(x)|x|^{\alpha/2}, {b>a>0}.$ Действительно, для
любого ${\delta>0}$ получим
$$
\sigma_f(-\delta,\delta]=\left\{\begin{array}{ll}
                                0&0<\delta\leq a\\
                                \frac{\delta^{\alpha+1}-a^{\alpha+1}}{\alpha+1} &a<\delta\leq b\\
                                \frac{b^{\alpha+1}-a^{\alpha+1}}{\alpha+1}
                                &b<\delta.
                               \end{array}\right.
$$
Отсюда заключаем, что
${\sigma_f(-\delta,\delta]
    \leq\frac{b^{\alpha+1}-a^{\alpha+1}}{b^\alpha}\frac{\delta^\alpha}{\alpha+1}}$
для всех ${\delta>0}.$ Обозначим через
${\tilde{\mathcal{X}}_\alpha}$ векторное подпространство линейных
комбинаций таких функций, т.е.
$$
\tilde{\mathcal{X}}_\alpha=\{f\in L_2(\mathbb{R}):\
f(x)=|x|^{\alpha/2}\sum_{i=1}^n c_i\chi_{(a_i,b_i)}, n\in\mathbb{N},
a_1<b_1\leq a_2<b_2\leq\ldots\leq a_n<b_n\}.
$$
С помощью этого подпространства покажем, что константу
${\frac{\rho(\alpha)}{2^\alpha}}$ во вторых утверждениях
теорем 1 и 2 нельзя уменьшить. А именно, справедливо следующее
утверждение.

%%%%%%%%%%%%%%%%%% Предложение о неулучшаемости константы

\begin{propos}\label{pr5}
Существуют гильбертово пространство $\mathcal{H},$ унитарная
группа $\{U^t\}_{t\in\R},$ действующая в $\mathcal{H},$ и
векторное подпространство ${\mathcal{X}\subseteq\mathcal{H}}$ со своей
нормой ${\|\cdot\|_\mathcal{X}},$ такие, что:

1) $\|P_{t,s}-P\|^2_{\mathcal{X}\to\mathcal{H}}=B(t-s)^{-\alpha}$
для всех ${t>s};$

2) для любого ${\varepsilon\in(0,1)}$ найдутся функция
${f\in\mathcal{X}}$ и ${\delta>0},$ для которых
$$
\sigma_{f-f^*}(-\delta,\delta]>
\varepsilon\frac{\rho(\alpha)}{2^\alpha}B\|f\|^2_\mathcal{X}\delta^\alpha.
$$
\end{propos}

\begin{proof}[Доказательство предложения~\ref{pr5}] Возьмем
${\mathcal{H}=L_2(\mathbb{R})},$ группу ${U^tf(x)=e^{itx}\cdot
f(x)}$ и подпространство ${\mathcal{X}=\tilde{\mathcal{X}}_\alpha},
\alpha\in[0,2].$ Зададим норму в $\mathcal{X}$ равенством
$$
\|f\|_\mathcal{X}:=\|f(x)|x|^{-\alpha/2}\|_2.
$$
Поскольку единственным неподвижным вектором относительно
рассматриваемой группы является нулевая функция, то ${P=0}.$ Найдем
соответствующую норму оператора $P_{t,s},$  где
$$
P_{t,s}f=\frac{1}{t-s}\int_s^tU^\tau
f(x)\,d\tau=\frac{e^{itx}-e^{isx}}{ix(t-s)}f(x).
$$
Полагая ${2\tau=t-s},$ получим
\begin{align*}
\|P_{t,s}f\|^2_\mathcal{H}&=\int_\mathbb{R}\frac{4\sin^2(x(t-s)/2)}{x^2(t-s)^2}|f(x)|^2\,dx
=\sum_{i=1}^n|c_i|^2\int_{a_i}^{b_i}F_{t-s}(x)|x|^\alpha\,dx=\\
&=\frac{1}{\tau^{\alpha+1}}\sum_{i=1}^n|c_i|^2\int_{\tau a_i}^{\tau
b_i}\frac{\sin^2y}{y^{2-\alpha}}\,dy
\leq\frac{1}{\tau^{\alpha}}\sum_{i=1}^n|c_i|^2(b_i-a_i)\sup\limits_{x\in(\tau
a_i,\tau b_i)}\frac{\sin^2y}{y^{2-\alpha}}\leq\\
&\leq\frac{1}{\tau^{\alpha}}\sup\limits_{x>0}\frac{\sin^2y}{y^{2-\alpha}}
\sum_{i=1}^n|c_i|^2(b_i-a_i)=\frac{2^\alpha}{\rho(\alpha)}\|f\|^2_{\mathcal{X}}(t-s)^{-\alpha}.
\end{align*}
Таким образом,
$\|P_{t,s}\|^2_{\mathcal{X}\to\mathcal{H}}\leq\frac{2^\alpha}{\rho(\alpha)}(t-s)^{-\alpha}.$
На самом деле будет равенство, поскольку супремум (в определении
операторной нормы) будет достигаться при ${\nu\to0+}$ на семействе
функций
${f_\nu(x)=|x|^{\alpha/2}\chi_{(\frac{r-\nu}{\tau},\frac{r+\nu}{\tau})}(x)},$
где $\rho(\alpha)=\frac{r^{2-\alpha}}{\sin^2r}.$

Зафиксируем теперь ${\varepsilon\in(0,1)}.$ Нужно предъявить функцию
$f\in\mathcal{X}$ и ${\delta>0}$ такие, что
${\sigma_f(-\delta,\delta]>\varepsilon\|f\|^2_\mathcal{X}\delta^\alpha}.$

Возьмем $f(x)=|x|^{\alpha/2}\chi_{(1,1+\nu)}$ и ${\delta=1+\nu};$
тогда (как мы уже выше вычисляли)
${\sigma_f(-\delta,\delta]=\frac{(1+\nu)^{\alpha+1}-1}{\alpha+1}}$ и
${\|f\|^2_\mathcal{X}=\nu}.$ Следовательно, остается доказать, что
найдется достаточно малое число ${\nu=\nu(\varepsilon)>0},$ такое,
что
$$
\frac{(1+\nu)^{\alpha+1}-1}{\alpha+1}>\varepsilon \nu(1+\nu)^\alpha.
$$
Применяя теорему Лагранжа о среднем значении к функции $x^{1+\alpha}$ на отрезке
$[1,1+\nu],$ для некоторого ${c\in(1,1+\nu)}$ последнее неравенство перепишем в виде
${c^\alpha>\varepsilon(1+\nu)^\alpha}.$  При ${\alpha=0}$ неравенство справедливо для
любого ${\nu>0}.$ При ${\alpha\in(0,2]}$ заведомо подходят все
${\nu>0},$ удовлетворяющие неравенству
${1\geq\varepsilon(1+\nu)^\alpha},$ т.е.
${\nu\leq\varepsilon^{-1/\alpha}-1}.$
\end{proof}

%%%%%%%%%%%%%%%%%%%%%%%%%%%%%%%%%%%%%%%%%%%%%%%%%%%%%%%%%%%%%%%%%%%%%%
%%%%%%%%%%%%%%%%%%%%%%%%%%%%%%%%%%%%%%%%%%%%%%%%%%%%%%%%%%%%%%%%%%%%%%

 \section{Приложения}

{\bf\thesection.1.} Как уже упоминалось во введении, в
статье~\cite{JP} была доказана следующая теорема (с приложениями к
получению оценок убывания временных средних решений уравнения
Шредингера, и линейных волновых уравнений).

\begin{theorem}\label{ThMorisse+BenArt}

    Пусть $\mathcal{X} \subset \mathcal{H}$ --- банахово пространство, которое плотно в $\mathcal{H}$, непрерывно в него вложено, т.е. $\|\cdot \|_\mathcal{H} \leq \|\cdot\|_\mathcal{X} $, и обладает следующими  свойствами:

    1) существуют число $r \in (0,1)$ и функция $\psi:[-r,r] \to \mathbb{R},$ которая является почти всюду строго положительной на отрезке $I_r=[-r,r]$ и ограничивает сверху плотность спектральной меры на множестве  $I_r\setminus \{0\}:$

    $|\frac{d}{d\lambda}(E(\lambda)f,g)| \leq \psi(\lambda)\|f\|_\mathcal{X}\|g\|_\mathcal{X} $ для всех $f,g \in \mathcal{X}$ для всех $\lambda \in I_r\setminus \{0\};$

    2) существует число $q>0$ такое что $|\lambda|^{-q}\psi(\lambda) \in L_1(I_r).$

    Тогда при $l=\min\{q,2\}$ имеет место степенная равномерная сходимость на пространстве $\mathcal{X}$ в теореме фон Неймана:
    \begin{center}
    $\|P_{t,-t}-P\|_{\mathcal{X} \to \mathcal{H}} \leq C{t^{-l/2}}$ для всех $t>1,$
    \end{center}
    где можно положить $C= \sqrt{\Psi_q(r)+\frac{1}{r^2}}, \, \Psi_q(r) = \int_{I_r}|\lambda|^{-q}\psi(\lambda)d\lambda.$
\end{theorem}

\begin{remark}\label{rm1}
     Из условий теоремы~\ref{ThMorisse+BenArt} вытекает
     локальная степенная с показателем $q$ оценка на спектральную меру,
     а именно: для всех $\delta \in (0,r)$
\begin{align*}
\sigma_{f-f^*}(-\delta,\delta]&=\sigma_f(-\delta,\delta]-\|f^*\|^2_\mathcal{H}=\\
&=\int \limits_{(-\delta,\delta ]\setminus\{0\}}d(E(\lambda)
f,f)=\int \limits_{(-\delta,\delta]\setminus\{0\}}\frac{d}{d\lambda}(E(\lambda)f,f)d\lambda\leq\\
&\leq\int\limits_{(-\delta,\delta]\setminus\{0\}}\psi(\lambda)\|f\|^2_\mathcal{X}d\lambda
=||f||^2_\mathcal{X}\int \limits_{(-\delta,\delta]\setminus\{0\}}|\lambda|^q|\lambda|^{-q}\psi(\lambda)d\lambda\leq\\
&\leq||f||^2_\mathcal{X}\delta^q\int_{I_r}|\lambda|^{-q}\psi(\lambda)d\lambda=\|f\|^2_\mathcal{X}\Psi_q(r)\delta^q.
\end{align*}
\end{remark}

%%%%%%%%%%%%%%%%%%%%%%%%%%%%%%%%%%

{\bf\thesection.2. Локальная степенная особенность спектральной
меры.} Теорема~\ref{ThMorisse+BenArt} и замечание~\ref{rm1} приводят
к естественной задаче построения аналогов теоремы~\ref{ThMain} и
предложения~\ref{pr3} в случае, когда имеется только локальная
степенная оценка для спектральной меры.
Лемма~\ref{lmLoc} позволяет легко перенести наши результаты
на аналогичный случай локальных оценок.

\begin{theorem}\label{ThMain+} Пусть $\alpha \geq 0, \mathcal{H}$ --- гильбертово пространство,
$\mathcal{X} \subseteq \mathcal{H}$ --- его векторное подпространство
со своей нормой  $\|\cdot\|_{\mathcal{X}},$ которое непрерывно в
него вложено, т.е. без ограничения общности считаем $\|\cdot\|_\mathcal{H} \leq
\|\cdot\|_\mathcal{X}$.

    Если для некоторого $r>0$ существует положительная константа $A,$ такая, что
    для всех $ f \in \mathcal{X}$ при всех $\delta \in (0,r)$ выполняется неравенство
    $$
    \sigma_{f-f^*}(-\delta,\delta]\leq A\|f\|^2_{\mathcal{X}}\delta^\alpha,
    $$
    то имеет место равномерная сходимость на пространстве $\mathcal{X}$
    в теореме фон Неймана: при $D=\max\{A,\frac{1}{r^\alpha}\}$

    1) в случае $\alpha\in [0,2)$ при всех $t>s$
    $$
    \|P_{t,s} - P\|^2_{\mathcal{X} \to \mathcal{H}} \leq B(t-s)^{-\alpha},
    \hbox { где можно положить } B = \frac{2^{\alpha+1}}{2-\alpha}D;
    $$

    2) в случае ${\alpha=2}$ при ${t\geq s+2}$
    $$
    \|P_{t,s} - P\|^2_{\mathcal{X} \to \mathcal{H}} \leq B\frac{\ln(t-s)}{(t-s)^2},
    \hbox { где можно положить } {B = 8D+\frac4{\ln 2}};
    $$

    3) в случае $\alpha>2$ при ${t\geq s+2}$
    $$
    \|P_{t,s} - P\|^2_{\mathcal{X} \to \mathcal{H}} \leq B(t-s)^{-2},
    \hbox { где можно положить } B=\frac{8}{\alpha-2}D+4.
    $$
    \end{theorem}

%%%%%%%%%%%%%%%%%%%%%%%%%%%%%%%%%%%%%%%%%%%%%%%%%%%
{\bf\thesection.3. Борелевские меры на прямой.}  Как показывает
замечание~\ref{rm1}, теорема~\ref{ThMorisse+BenArt} при ${q\neq2}$
является, с точностью до значений получающихся констант, очень
частным случаем теоремы~\ref{ThMain+}.
Например, тут не требуется ни банаховости вложенного нормированного
пространства $\mathcal{X},$ ни плотности вложения, ни (локальной)
абсолютной непрерывности спектральных мер относительно меры Лебега.
Кроме того, следуя формулировкам фон Неймана~\cite{Neum}, наши теоремы
работают для общих усреднений $P_{t,s}$ (для  $P_{t,0},$ например
--- а не только для $P_{t,-t},$ как в~\cite{JP}), и содержат
утверждения о неулучшаемости получаемых асимптотик скоростей
сходимости.

При этом оценка в случае ${q=2}$ в условиях
теоремы~\ref{ThMorisse+BenArt} оказывается точнее, чем в пункте~2
теоремы~\ref{ThMain+}. Оказывается, дело тут в специфической
постановке задачи: в~\cite{JP} особенность спектральной меры в нуле
оценивалась в терминах конечности интеграла от ее плотности $\phi
(x)$, помноженной на степенную функцию $|x|^{-q}.$ В заключительной
части статьи проанализируем взаимосвязи этих двух и других возможных
подходов к постановке задачи.

Рассуждения проведем в общем случае, введя для всех ${q\geq0}$
следующие классы конечных борелевских мер $\mu$ на вещественной
прямой.

Пусть $\mathcal{K}^{q}_1$ содержит все конечные борелевские
меры $\mu$ на $\mathbb{R}$, абсолютно непрерывные относительно меры
Лебега на множестве $[-r,r]\setminus\{0\}$ для некоторого
${r\in(0,1)},$ плотность которых ${\phi(x)}$ удовлетворяет
условию ${\int_{-r}^r|x|^{-q}\phi(x)\,dx<\infty}.$

Класс $\mathcal{K}^{q}_2$ содержит все конечные борелевские меры $\mu$
на $\mathbb{R}$ для которых
$$
{\int_{\mathbb{R}\setminus\{0\}}|x|^{-q}\,d\mu(x)<\infty}.
$$

Класс $\mathcal{K}^{q}_3$ содержит все конечные борелевские меры $\mu$
на $\mathbb{R}$ для которых
$$
\int_{\mathbb{R}\setminus\{0\}}\frac{\sin^2(\tau x)}{(\tau
x)^2}\,d\mu(x)\leq B\tau^{-q}
$$
для всех ${\tau>0}$ и некоторой константы $B>0.$

Класс $\mathcal{K}^{q}_4$ содержит все конечные борелевские меры $\mu$
на $\mathbb{R}$ для которых
$$
{\mu\{(-\delta,\delta]\setminus\{0\}\}\leq A\delta^\alpha}
$$
для всех ${\delta>0}$ и некоторой константы $A>0.$

Ясно, что классы
${\mathcal{K}^{0}_2=\mathcal{K}^{0}_3=\mathcal{K}^{0}_4}$ и
совпадают со множеством всех конечных борелевских мер на
$\mathbb{R}.$ Класс ${\mathcal{K}^{0}_1}$ состоит из всех локально
(в проколотой окрестности нуля) абсолютно непрерывных мер. Как легко
заметить, в теореме~\ref{ThMorisse+BenArt} рассматриваются
спектральные меры из класса ${\mathcal{K}^{q}_1}$ и по сути
доказывается, что эти меры лежат в классе
${\mathcal{K}^{\min\{q,2\}}_3}.$ В наших же теоремах рассматриваются
спектральные меры из класса~${\mathcal{K}^{q}_4}$ и тоже
доказывается, что они лежат в классе ${\mathcal{K}^{q}_3},$ но
только для ${q\neq2}.$ Для ${q=2}$ вложение классов, как показано в
следующей теореме, имеет место в другую сторону.

%%%%%%% Теорема о соотношении классов борелевских мер на прямой %%%%%%%%%%%%%%%%%%%%

\begin{theorem}\label{ThMeasures}  Cправедливы следующие включения:

$(1)\ {\mathcal{K}^{q}_1\subsetneq\mathcal{K}^{q}_2}$ для любого
${q>0};$

$(2)\ {\mathcal{K}^{q}_2\subsetneq\mathcal{K}^{\min\{q,2\}}_3}$
для любого ${q\neq2},$ и
${\mathcal{K}^{2}_2=\mathcal{K}^{2}_3};$

$(3)\ \mathcal{K}^{q}_2\subsetneq\mathcal{K}^{q}_4$
 и ${\mathcal{K}^{q+p}_4\subsetneq\mathcal{K}^{q}_2}$ для любых ${q,p>0};$

$(4)\ {\mathcal{K}^{q}_3=\mathcal{K}^{q}_4}$ для ${q\in[0,2)},$
и
${{\mathcal{K}^{q}_4\subsetneq\mathcal{K}^{2}_3}\subsetneq\mathcal{K}^{2}_4}$
для ${q>2};$

$(5)\ {\mathcal{K}^{q}_3}$ для любого ${q>2}$ состоит из атомических
мер с носителем в точке $0$.
\end{theorem}

\begin{proof}[Доказательство теоремы~\ref{ThMeasures}]
(1) Пусть мера ${\mu\in\mathcal{K}^{q}_1};$ тогда
\begin{align*}
\int_{\mathbb{R}\setminus\{0\}}|x|^{-q}\,d\mu(x)&
=\int_{[-r,r]\setminus\{0\}}|x|^{-q}\,d\mu(x)+\int_{|x|>r}|x|^{-q}\,d\mu(x)\leq\\
&\leq\int_{[-r,r]\setminus\{0\}}|x|^{-q}\phi(x)\,dx+\int_{|x|>r}r^{-q}\,d\mu(x)\leq\\
&=\int_{[-r,r]\setminus\{0\}}|x|^{-q}\phi(x)\,dx+r^{-q}\mu(\mathbb{R})<\infty,
\end{align*}
т.е. ${\mu\in\mathcal{K}^{q}_2}.$ В качестве меры
${\mu\in\mathcal{K}^{q}_2\setminus\mathcal{K}^{q}_1}$ можно
взять меру ${d\mu(x)=x^q\chi_{[0,1]}(x)dk(x)},$ где $dk(x)$ --- мера
на отрезке $[0,1],$ порожденная функцией Kантора $k(x).$

(2) Пусть мера ${\mu\in\mathcal{K}^{q}_2}$ для ${q\in[0,2]};$
тогда для любого ${\tau>0}$
\begin{align*}
\int_{\mathbb{R}\setminus\{0\}}\frac{\sin^2(\tau x)}{(\tau
x)^2}\,d\mu(x)&=\int_{\mathbb{R}\setminus\{0\}}\frac{\sin^2(\tau
x)}{(\tau |x|)^{2-q}}|\tau
x|^{-q}\,d\mu(x)\leq\\
&\leq\sup\limits_{y>0}\frac{\sin^2y}{y^{2-q}}\int_{\mathbb{R}\setminus\{0\}}|\tau
x|^{-q}\,d\mu(x)=
\frac{\tau^{-q}}{\rho(q)}\int_{\mathbb{R}\setminus\{0\}}|x|^{-q}\,d\mu(x).
\end{align*}

Пусть теперь ${q>2};$ тогда для любого ${\tau>0}$
\begin{align*}
\int_{\mathbb{R}\setminus\{0\}}\frac{\sin^2(\tau x)}{(\tau
x)^2}\,d\mu(x)&\leq\int_{[-1,1]\setminus\{0\}}\frac{1}{(\tau
x)^2}\,d\mu(x)+\int_{|x|>1}\frac{1}{(\tau x)^2}\,d\mu(x)\leq\\
&\leq\int_{[-1,1]\setminus\{0\}}\frac{1}{\tau^2|x|^q}\,d\mu(x)
+\int_{|x|>1}\frac{1}{\tau^2}\,d\mu(x)\leq\\
&\leq\tau^{-2}\left(\int_{\mathbb{R}\setminus\{0\}}|x|^{-q}\,d\mu(x)+\mu(\mathbb{R})\right).
\end{align*}
Таким образом, для любого ${q\geq0}$ справедливо включение
${\mathcal{K}^{q}_2\subset\mathcal{K}^{\min\{q,2\}}_3}.$ При
этом для ${q\neq2}$ мера ${d\mu(x)=x^{q-1}\chi_{[0,1]}dx}$ лежит в
${\mathcal{K}^{\min\{q,2\}}_3\setminus\mathcal{K}^{q}_2}.$

Покажем теперь равенство
${\mathcal{K}^{2}_2=\mathcal{K}^{2}_3}.$ Для этого достаточно
показать обратное включение. Пусть ${\mu\in\mathcal{K}^{2}_3},$
т.е.
$$
{\int_{\mathbb{R}\setminus\{0\}}\frac{\sin^2(\tau
x)}{x^2}\,d\mu(x)\leq B}
$$
для всех ${\tau>0}$ и некоторого ${B>0}.$
Интегрируя это неравенство по $\tau$ в интервале ${(0,T)}, T>0$ и
переходя ко второму повторному интегралу (по теореме Тонелли мы
можем это сделать), получим
\begin{align*}
2B\geq&\frac{2}{T}\int_0^T\left(\int_{\mathbb{R}\setminus\{0\}}\frac{\sin^2(\tau
x)}{x^2}\,d\mu(x)\right)d\tau=\frac{2}{T}\int_{\mathbb{R}\setminus\{0\}}x^{-2}\left(\int_0^T\sin^2(\tau
x)\,d\tau\right)d\mu(x)=\\
&=\frac{1}{T}\int_{\mathbb{R}\setminus\{0\}}x^{-2}\left(\int_0^T1-\cos(2\tau
x)\,d\tau\right)d\mu(x)=\int_{\mathbb{R}\setminus\{0\}}x^{-2}\left(1-\frac{\sin(2Tx)}{2Tx}\right)\,d\mu(x).
\end{align*}
Переходя к нижнему пределу при ${T\to+\infty},$ по лемме Фату
получаем необходимое неравенство
${2B\geq\int_{\mathbb{R}\setminus\{0\}}x^{-2}\,d\mu(x)},$ т.е.
${\mu\in\mathcal{K}^{2}_2}.$ Отметим, что идею этого рассуждения
мы почерпнули из~\cite{Rob} (см. также~\cite{Leo}).

(3) Пусть ${\mu\in\mathcal{K}^{q}_2}$ для ${q\geq0};$ тогда для
любого ${\delta>0}$
\begin{align*}
\mu\{(-\delta,\delta]\setminus\{0\}\}&
=\int\limits_{(-\delta,\delta]\setminus\{0\}}d\mu(x)
=\int\limits_{(-\delta,\delta]\setminus\{0\}}|x|^q|x|^{-q}d\mu(x)\leq\\
&\leq\delta^q\int\limits_{(-\delta,\delta]\setminus\{0\}}|x|^{-q}d\mu(x)
\leq\delta^q\int\limits_{\mathbb{R}\setminus\{0\}}|x|^{-q}d\mu(x),
\end{align*}
т.е. ${\mu\in\mathcal{K}^{q}_4}.$

Пусть теперь мера
${\mu\in\mathcal{K}^{q+p}_4}$ для ${q,p>0};$ тогда
\begin{align*}
\int_{\mathbb{R}\setminus\{0\}}|x|^{-q}d\mu(x)&
=\int_{\mathbb{R}\setminus\{0\}}q\int_0^{|x|^{-1}}t^{q-1}dtd\mu(x)=
q\int_0^{+\infty}t^{q-1}\int_{0<|x|<1/t}d\mu(x)dt=\\
&=q\int_0^{+\infty}t^{q-1}\mu\{(-1/t,1/t)\setminus\{0\}\}dt\leq\\
&\leq
q\mu(\mathbb{R})\int_0^{1}t^{q-1}dt+qA\int_1^{+\infty}t^{q-1}t^{-p-q}dt
=\mu(\mathbb{R})+A\frac{q}{p}<\infty,
\end{align*}
т.е. ${\mu\in\mathcal{K}^{q}_2}.$ В качестве меры
${\mu\in\mathcal{K}^{q}_2\setminus\bigcup_{p>0}\mathcal{K}^{q+p}_4}$
можно взять меру $${d\mu(x)=x^{q-1}|\ln
x|^{-2}\chi_{[0,1/2]}(x)dx}.$$

(4) Доказательство равенства классов
${\mathcal{K}^{q}_3=\mathcal{K}^{q}_4}$ для ${q\in[0,2)}$
дословно соответствует доказательству теоремы~\ref{Th1}, где
рассматривались спектральные меры $\mu=\sigma_{f-f^*}.$ Далее, для
${q=2+p>2},$ используя уже доказанные включения в пунктах (1) и (2),
получаем
$$
\mathcal{K}^{q}_4\subsetneq\mathcal{K}^{2+p/2}_2\subseteq\mathcal{K}^{2}_3.
$$
Включение ${\mathcal{K}^{2}_3\subset\mathcal{K}^{2}_4}$ также
уже было доказано в теореме~\ref{Th1} с использованием
леммы~\ref{lmLowerEst}.

(5) Доказываемое утверждение немедленно следует из следующей леммы~\ref{lmLast}.
\end{proof}

\begin{lemma}\label{lmLast}
Пусть ${\int_{\mathbb{R}\setminus\{0\}}\frac{\sin^2(\tau
x)}{x^2}\,d\mu(x)=o(1)}$ при ${\tau\to+\infty}.$ Тогда
${\mu(\mathbb{R}\setminus\{0\})=0}.$
\end{lemma}

\begin{proof}[Доказательство леммы~\ref{lmLast}] В терминах нормы в пространстве ${L_2(\mathbb{R}\setminus\{0\},\mu)}$ условие леммы
переписывается как ${\left\|\frac{\sin(\tau x)}{x}\right\|_2=o(1)}$
при ${\tau\to+\infty}.$ Тогда для любого ${\tau_0>0}$ будут
выполняться также следующие три асимптотические соотношения при
${\tau\to+\infty}:$
$$
\left\|\frac{\sin(\tau x)\sin(\tau_0x)}{x}\right\|_2=o(1),\ \
\left\|\frac{\sin(\tau x)\cos(\tau_0x)}{x}\right\|_2=o(1),\ \
\left\|\frac{\sin((\tau+\tau_0) x)}{x}\right\|_2=o(1).
$$
Из второго и третьего соотношений, раскладывая синус суммы и используя
неравенство треугольника для нормы, получаем
$$
{\left\|\frac{\cos(\tau x)\sin(\tau_0x)}{x}\right\|_2=o(1)}.
$$
Возводя это в квадрат и сложив с первым из трех соотношений, получаем равенство
$$
\int_{\mathbb{R}\setminus\{0\}}\frac{\sin^2(\tau_0x)}{x^2}d\mu(x)=0,
$$
справедливое для любого ${\tau_0>0}.$ Следовательно,
$$
{\int_{\mathbb{R}\setminus\{0\}}\frac{\sin^2(x)+\sin^2(\pi
x)}{x^2}d\mu(x)=0}.
$$
Поскольку подынтегральная функция положительна,
то ${\mu(\mathbb{R}\setminus\{0\})=0}.$
\end{proof}

{\bf\thesection.4. Максимально возможная скорость сходимости:
степенная с $\alpha=2$.}
%Приведем несколько простых важных следствий из теоремы~\ref{ThMeasures}.
Следующее замечание хорошо известно (и может быть доказано и для действий групп
линейных изометрий в негильбертовых пространствах  --- см., например,~\cite{Se14}).

\begin{remark}\label{rm2}
Из леммы~\ref{lmLast} и интегрального представления (1) немедленно следует,
что скорость сходимости $O((t-s)^{-2})$ в эргодической теореме фон Неймана
является максимально возможной, т.е. что асимптотическое соотношение
$\|P_{t,s}f-f^*\|_\mathcal{H}^2= o((t-s)^{-2})$ при $t-s\rightarrow\infty$
выполняется только в вырожденном случае $f=f^*.$
\end{remark}

С учетом замечания~\ref{rm2}, из доказанного в пункте~2
теоремы~\ref{ThMeasures} равенства
${\mathcal{K}^{2}_2=\mathcal{K}^{2}_3}$ немедленно получается
следующий спектральный критерий максимально возможной скорости
сходимости в эргодической теореме фон Неймана:
$\|P_{t,s}f-f^*\|_\mathcal{H}^2= O(t-s)^{-2}$ при
$t-s\rightarrow\infty$ тогда и только тогда, когда конечен интеграл
$\int_{\mathbb{R}}x^{-2}\,d\sigma_{f-f^*}(x).$ Мы сформулируем этот
результат в трех вариантах: для одномерных подпространств
(теорема~\ref{Th2} --- аналог  теоремы~\ref{Th1}), для многомерных
подпространств (теорема~\ref{Th2+} --- аналог теоремы~\ref{ThMain}),
и теорема~\ref{Th2++}, позволяющая дать локальный вариант
теоремы~\ref{Th2+}.

\begin{theorem}\label{Th2}
    Зафиксируем ${f \in \mathcal{H}}$. Тогда:

    1. Если
    $$
    \int_{\mathbb{R}}x^{-2}\,d\sigma_{f-f^*}(x)=A<\infty,
    $$
    то скорость сходимости эргодических средних $P_{t,s}f$ --- степенная с
    показателем степени 2, т.е. при всех $t > s$
    $$
    \|P_{t,s}f-f^{*}\|^2_{\mathcal{H}} \leq B(t-s)^{-2},
    $$
    где можно положить ${B=4A}.$

    2. Если скорость сходимости эргодических средних $P_{t,s}f$ ---
    степенная с показателем степени~2, т.е. если для некоторой
    положительной константы $B$ при всех ${t > s}$ выполняется неравенство
    $$
    \|P_{t,s}f-f^*\|_\mathcal{H}^2\leq B(t-s)^{-2},
    $$
    то
    $$
    \int_{\mathbb{R}}x^{-2}\,d\sigma_{f-f^*}(x)\leq A,
    $$
    где можно положить ${A=8B}.$
\end{theorem}

\begin{proof}[Доказательство теоремы~\ref{Th2}]
Если $\int_{\mathbb{R}}x^{-2}\,d\sigma_{f-f^*}(x)= A<\infty,$ то
(см. доказательство пункта~2 теоремы~\ref{ThMeasures}) по представлению (1)
при всех ${t>s}$
$$
\|P_{t,s}f-f^*\|_\mathcal{H}^2=
{\int_{\mathbb{R}}
\left(\frac{\sin\frac{(t-s)x}{2}}{\frac{(t-s)x}{2}}\right)^2
d\sigma_{f-f^*}(x)}\leq
\frac{(\frac{t-s}{2})^{-2}}{\rho(2)}\int_{\mathbb{R}}x^{-2}\,d\sigma_{f-f^*}(x)
= 4A(t-s)^{-2}.
$$

С другой стороны, если для некоторой константы $B>0$ при всех
${t>s}$ будет
$$
\|P_{t,s}f-f^*\|_\mathcal{H}^2
\leq B(t-s)^{-2},
$$
то это неравенство с учетом  представления (1) переписывается в виде
$$
\int_{\mathbb{R}}\left(\frac{\sin\frac{(t-s)x}{2}}{\frac{(t-s)x}{2}}\right)^2
    d\sigma_{f-f^*}(x)
\leq4B\left(\frac{t-s}2\right)^2,
$$
и по оценкам доказательства пункта 2 теоремы~\ref{ThMeasures}
получаем $ \int_{\mathbb{R}}x^{-2}\,d\sigma_{f-f^*}(x)\leq 8B$.
\end{proof}

\begin{theorem}\label{Th2+}
    Пусть $\mathcal{H}$ --- гильбертово пространство,
    $\mathcal{X} \subseteq \mathcal{H}$ --- его векторное подпространство
    со своей нормой  $\|\cdot\|_{\mathcal{X}}$. Тогда:

    1. Если существует положительная константа $A,$ такая, что
    для всех $f \in \mathcal{X}$ выполняется неравенство
    $$
    \int_{\mathbb{R}}x^{-2}\,d\sigma_{f-f^*}(x)\leq A\|f\|^2_{\mathcal{X}},
    $$
    то имеет место степенная с показателем~2 равномерная сходимость на пространстве $\mathcal{X}$ в теореме фон Неймана: при всех ${t>s}$
    $$
    \|P_{t,s} - P\|^2_{\mathcal{X} \to \mathcal{H}} \leq B(t-s)^{-2},
    $$
    где можно положить ${B=4A}.$

    2. Если имеет место степенная с показателем~2 равномерная сходимость на пространстве
    $\mathcal{X}$ в теореме фон Неймана,
    т.е. для некоторой положительной константы $B$ при всех ${t>s}$ выполняется неравенство
    $$
    \|P_{t,s} - P\|^2_{\mathcal{X} \to \mathcal{H}} \leq B(t-s)^{-2},
    $$
    то для всех $f\in \mathcal{X}$
    $$
    \int_{\mathbb{R}}x^{-2}\,d\sigma_{f-f^*}(x)
    \leq A \|f\|^2_{\mathcal{X}},
    $$
    где можно положить ${A=8B}.$
\end{theorem}

\begin{proof}[Доказательство теоремы~\ref{Th2+}] В условиях первой части теоремы,
    из утверждения первой части теоремы~\ref{Th2} немедленно получаем, что для всех ${f\in\mathcal{X}}$ при всех ${t > s}$
    $$
    \|(P_{t,s} - P)f\|^2_{\mathcal{H}}
    \leq 4A\|f\|^2_{\mathcal{X}}(t-s)^{-2};
    $$
    поэтому
    $$
    \|P_{t,s} - P\|^2_{\mathcal{X} \to \mathcal{H}}=
    \sup\limits_{f \in \mathcal{X} : f \neq 0} \frac{\|(P_{t,s} - P)f\|^2_\mathcal{H}}{\|f\|^2_{\mathcal{X}}} \leq 4A (t-s)^{-2}.
    $$
    В условиях второй части теоремы, для всех ${f\in\mathcal{X}}$
    при всех ${t>s}$
    $$
    \frac{\|(P_{t,s} - P)f\|^2_\mathcal{H}}{\|f\|^2_{\mathcal{X}}} \leq B (t-s)^{-2},
    $$
    и из утверждения второй части теоремы~\ref{Th2} немедленно следует, что
    $$
    \int_{\mathbb{R}}x^{-2}\,d\sigma_{f-f^*}(x)
    \leq 8B\|f\|^2_{\mathcal{X}}.
    $$
\end{proof}
\begin{theorem}\label{Th2++} Пусть $\mathcal{H}$ --- гильбертово пространство,
    $\mathcal{X} \subseteq \mathcal{H}$ --- его векторное подпространство
    со своей нормой  $\|\cdot\|_{\mathcal{X}},$ которое непрерывно в
    него вложено, т.е. без ограничения общности считаем $\|\cdot\|_\mathcal{H} \leq
    \|\cdot\|_\mathcal{X}$.

    Если для некоторого $r>0$ существует положительная константа $A,$ такая, что
    для всех $ f \in \mathcal{X}$
    $$
    \int_{(-r,r]}x^{-2}\,d\sigma_{f-f^*}(x)
    \leq A\|f\|^2_{\mathcal{X}},
    $$
    то имеет место степенная с показателем 2 равномерная сходимость
    на пространстве $\mathcal{X}$ в теореме фон Неймана: при всех $t>s$
    $$
    \|P_{t,s} - P\|^2_{\mathcal{X} \to \mathcal{H}} \leq B(t-s)^{-2},
    $$
    где можно положить ${B=4\left(A+\frac{1}{r^2}\right)}.$
    \end{theorem}
\begin{proof}[Доказательство теоремы~\ref{Th2++}]
По представлению (1) при всех ${t>s}$

\begin{align*}
\|P_{t,s}f-f^*\|_\mathcal{H}^2&=\int\limits_{\mathbb{R}}
%F_{t-s}(x)
\left(\frac{\sin\frac{(t-s)x}{2}}{\frac{(t-s)x}{2}}\right)^2
d\sigma_{f-f^*}(x)\leq
\frac4{(t-s)^{2}}\int\limits_{\mathbb{R}}x^{-2}\,d\sigma_{f-f^*}(x)=\\
&=\frac4{(t-s)^{2}}\int\limits_{(-r,r]} x^{-2}\,d\sigma_{f-f^*}(x)
+\frac4{(t-s)^{2}}\int\limits_{(-\infty,-r]\cup(r,\infty)}
x^{-2}\,d\sigma_{f-f^*}(x)\leq\\
&\leq\frac{4}{(t-s)^{2}}\left(A\|f\|^2_\mathcal{X}+\frac{\|f\|^2_\mathcal{H}}{r^2}\right)
\leq
\frac{4\|f\|^2_\mathcal{X}}{(t-s)^{2}}\left(A+\frac{1}{r^2}\right),
\end{align*}
%Отсюда заключаем, что ${\|P_{t,s} - P\|^2_{\mathcal{X} \to
%\mathcal{H}} \leq 4\left(A+\frac{1}{r^2}\right)(t-s)^{-2}}$.
что и требовалось.
\end{proof}

{\bf\thesection.5.} Сделаем несколько заключительных замечаний.
\begin{remark}\label{rm5}
Следуя доказательствам пунктов (1) и (2) теоремы~\ref{ThMeasures},
можно слегка улучшить константу, возникающую в
теореме~\ref{ThMorisse+BenArt} при ${q\in[0,2]}$.
\end{remark}
А именно, в условиях теоремы~\ref{ThMorisse+BenArt} при всех
${q\in[0,2]}$ для любой ${f\in\mathcal{X}}$ будет
$$
\|P_{t,-t}f-Pf\|^2_{\mathcal{H}}=
\int_{\mathbb{R}\setminus\{0\}}\frac{\sin^2(tx)}{(tx)^2}d\sigma_f(x)\leq
\frac{t^{-q}}{\rho(q)}\int_{\mathbb{R}\setminus\{0\}}|x|^{-q}\,d\sigma_f(x)\leq
$$
$$
\leq\frac{t^{-q}}{\rho(q)}
\left(
\|f\|^2_\mathcal{X}\int_{[-r,r]\setminus\{0\}}|x|^{-q}\psi(x)\,dx+r^{-q}\sigma_f(\mathbb{R})
\right)=
\frac{t^{-q}}{\rho(q)}\left(\|f\|^2_\mathcal{X}\Psi_q(r)+r^{-q}\|f\|^2_\mathcal{H}\right),
%\leq\frac{\Psi_q(r)+r^{-q}}{\rho(q)}\frac{\|f\|^2_\mathcal{X}},
$$
т.е. при всех ${t>0}$
$$
\|P_{t,-t}-P\|_{\mathcal{X} \to \mathcal{H}}\leq
\sqrt{\frac{\Psi_q(r)+r^{-q}}{\rho(q)}}{t^{-q/2}}.
$$
Теорема~\ref{Th2++} при $q=2$ дает точно такую же оценку.
%$$
%\|P_{t,-t}-P\|_{\mathcal{X} \to \mathcal{H}}
%\leq \sqrt{\max\{\Psi_q(r), r^{-2}\}}t^{-1}.
%$$
\begin{remark}\label{rm4} Рассуждения, проведенные при доказательстве теоремы~\ref{Th2} выше,
    позволяют ввести на пространстве $\mathcal{Y}$ из формулировки
    теоремы~\ref{ThCharacterizationOfSpaces} норму
    ${\||\cdot\||_{\mathcal{Y}}}$, эквивалентную норме
    ${\|\cdot\|_{\mathcal{Y}}}$, основанную на особенности спектральной
    меры в нуле, a именно:
    $${\||f\||^2_{\mathcal{Y}}=\int_\mathbb{R}x^{-2}\,d\sigma_f(x)}.$$
\end{remark}

Важным инструментом исследования конечных борелевских мер~$\mu$ на
прямой является преобразование
Фурье~${\hat{\mu}(t)=\int_\mathbb{R}e^{itx}\,d\mu(x)},
{t\in\mathbb{R}}.$ По его асимптотике на бесконечности можно судить
о принадлежности меры введенным выше классам. Приведем примеры таких
утверждений.

\begin{remark}\label{rm6}  Каждая конечная борелевская мера~$\mu,$
для которой преобразование Фурье  при ${t\to+\infty}$ имеет
степенную скорость сходимости порядка $O(t^{-q})$ для ${q\in[0,1)}$,
будет принадлежать классу~${\mathcal{K}^q_3}.$ Для спектральных мер
полупотоков этот результат был получен в~\cite[теорема~2]{JK};
доказательство для произвольных мер проводится аналогично.

Если же ${\frac{1}{T}\int_0^T|\hat{\mu}(t)|^2\,dt}$ имеет степенную
асимптотику ${O(T^{-q})}$ при ${T\to\infty}$ для ${q\in[0,1]},$ то мера
$\mu$ будет равномерно $q/2$-гельдеровской (т.е. найдется константа
$C>0$ такая, что для любого интервала $I$ с мерой Лебега $|I|<1$
будет ${\mu(I)\leq C|I|^{q/2}}$)~\cite[теорема~3.1]{Last96}. В
частности, такая мера лежит в классе~${\mathcal{K}^{q/2}_4}.$
\end{remark}

Работа выполнена в рамках государственного задания ИМ СО РАН (проект
№ FWNF-2022-0004).

%ЖЖЖЖЖЖЖЖЖЖЖЖЖЖЖЖЖЖЖЖЖЖЖЖЖЖЖЖЖЖЖЖЖЖЖЖЖЖЖЖЖЖЖЖЖЖЖЖЖЖЖЖЖЖЖЖЖЖЖЖЖЖЖЖЖЖЖЖЖЖЖЖЖЖЖЖЖ
%ЖЖЖЖЖЖЖЖЖЖЖЖЖЖЖЖЖЖЖЖЖЖЖЖЖЖЖЖЖЖЖЖЖЖЖЖЖЖЖЖЖЖЖЖЖЖЖЖЖЖЖЖЖЖЖЖЖЖЖЖЖЖЖЖЖЖЖЖЖЖЖЖЖЖЖЖЖ

\newpage

\begin{center}
АННОТАЦИЯ
\end{center}

\

Качуровский А.Г., Подвигин И.В., Тодиков В.Э. {\sc Равномерная сходимость на подпространствах в
эргодической теореме фон Неймана с непрерывным временем.}

\

Рассматривается степенная равномерная (в операторной норме) сходимость на векторных подпространствах со своими нормами в эргодической теореме фон Неймана с непрерывным временем. Найдены все возможные показатели степени рассматриваемой степенной сходимости; для каждого из этих показателей даны спектральные критерии такой сходимости и получено полное описание всех таких подпространств. Равномерная сходимость на всем пространстве имеет место
лишь в тривиальных случаях, что объясняет интерес к равномерной сходимости
именно на подпространствах.

Кроме того, попутно обобщены и уточнены старые оценки скоростей сходимости
в эргодической теореме фон Неймана для (полу)потоков.

\

Ключевые слова: эргодическая теорема фон Неймана; скорости сходимости в эргодических теоремах; степенная равномерная сходимость.

\

УДК 517.987+519.214

%------------------------------------------------------------------------------------------------

\newpage

\begin{center}
    ANNOTATION
\end{center}

\

Kachurovskii A.G., Podvigin I.V., Todikov V.E. {\sc Uniform convergence \\ on subspaces in von Neumann's ergodic theorem with continuous time.}

\

Power uniform (in the operator norm) convergence on vector subspaces
with their own norms in von Neumann's ergodic theorem with
continuous time is considered. All possible exponents of the
considered power convergence are found; for each of these exponents,
spectral criteria for such convergence are given and a complete
description of all such subspaces is obtained. Uniform convergence
over the entire space takes place only in trivial cases, which
explains the interest in the uniform convergence just on subspaces.

In addition, along the way, the old convergence rate estimates in the von Neumann ergodic theorem for (semi)flows are generalized and refined.

\

Key words: von Neumann's ergodic theorem; rates of convergence in ergodic theorems; power uniform convergence.

\

MSC2020: Primary 37A30; Secondary 37A10, 47A35, 60G10.

%ЖЖЖЖЖЖЖЖЖЖЖЖЖЖЖЖЖЖЖЖЖЖЖЖЖЖЖЖЖЖЖЖЖЖЖЖЖЖЖЖЖЖЖЖЖЖЖЖЖЖЖЖЖЖЖЖЖЖЖЖЖЖЖЖЖ
%ЖЖЖЖЖЖЖЖЖЖЖЖЖЖЖЖЖЖЖЖЖЖЖЖЖЖЖЖЖЖЖЖЖЖЖЖЖЖЖЖЖЖЖЖЖЖЖЖЖЖЖЖЖЖЖЖЖЖЖЖЖЖЖЖЖ

\end{document}